\documentclass[10pt]{amsart}
\usepackage{amsmath,amsfonts,amsthm,amstext,amscd,amssymb,mathtools}

\usepackage{hyperref,cleveref}
\usepackage{enumerate}
\usepackage[final]{showlabels}
\usepackage{stmaryrd}
\def\rm#1{\mathrm{#1}}
\def\cal#1{\mathcal{#1}}
\def\bb#1{\mathbb{#1}}
\def\lie#1{\mathfrak{#1}}
\def\lr#1{\left\langle #1\right\rangle}
\usepackage[all]{xy}

\def\co{\colon}
\newcommand{\halmos}{\hfill $\;\;\;\Box$\\}

\newcommand{\h}{\tfrac{1}{2}}
\newcommand{\I}{{e}}
\newcommand{\HH}{{\cal H}^{\bb C}_{\I}}

\DeclareMathOperator{\id}{id}
\DeclareMathOperator{\ad}{ad}

\DeclareMathOperator{\Hol}{Hol}
\DeclareMathOperator{\Ad}{Ad}
\DeclareMathOperator{\End}{End}

\newtheorem{theorem}{Theorem}[section]

\newtheorem{conj}{Conjecture}

 \newtheorem{cor}[theorem]{Corollary}
 \newtheorem{lem}[theorem]{Lemma}
 \newtheorem{prop}[theorem]{Proposition}
 \newtheorem{claim}[theorem]{Claim}

\theoremstyle{definition}

\theoremstyle{remark}
\newtheorem{rem}[theorem]{Remark}

\begin{document}

\title[Riemannian Foliations on Compact Lie Groups]{Totally Geodesic Riemannian Foliations on Compact Lie Groups}

\author{Llohann D. Speran\c ca}
\thanks{This work was partially supported by CNPq grant number 404266/2016-9 and FAPESP grant number 2017/10892-7.}
\address{Universidade Federal de S\~ao Paulo, ICT\\
	Av. Cesare Monsueto Giulio Lattes, 1211 - Jardim Santa Ines I\\
	CEP 12231-280 \\
	S\~ao Jos\'e dos Campos, SP, Brazil}
\email{speranca@unifesp.br}



\begin{abstract}
	In 86 Ranjan questioned whether a submersion $\pi\co G\to B$ from  a compact simple Lie group with bi-invariant metric is a coset foliation or not, provided the submersion is Riemannian with totally geodesic fibers. Here we answer this question affirmatively, even when the submersion is defined only in an open subset of $G$ (assuming suitable compactness hypothesis).
%	Let $\cal F$ be a Riemannian foliation with connected totally geodesic leaves on a connected compact Lie group with bi-invariant metric $G$. We answer a question of A. Ranjan  proving that, up to isometry, the leaves of  $\cal F$ coincide with the cosets of a fixed subgroup.
\end{abstract}

\maketitle
\section{Introduction}\label{sec:1}
The present work is dedicated to the simple question: how to fill a given geometric space with a geometric pattern? Or,  following Thurston \cite{thurston1974construction}: how to construct a manifold out of  stripped fabric?  
For instance, starting with a Lie group $G$, we could use its algebraic structure to construct a pattern: any Lie subgroup $H<G$ induces a decomposition of $G$ by both  right cosets, $\cal F^+_H=\{gH~|~h\in G\}$, and left cosets, $\cal F^-_H=\{Hg~|~h\in G\}$. Such decompositions are called as \textit{coset  foliations}.

In general, a \textit{foliation}\footnote{Only non-singular foliations with connected leaves are considered in this paper.} $\cal F$ on $M$ is the decomposition of $M$ into the integrable maximal submanifolds of an involutive subbundle  $\cal V=T\cal F\subseteq TM$. Such submanifolds are called \textit{leaves}. Existence, obstructions and classifications of foliations are deep topological subjects (see e.g. Haefliger \cite{haefliger} and Thurston \cite{thurston1976existence,thurston1974theory,thurston1974construction}) and they acquire a  geometric flavor by imposing distance rigidity between leaves: a foliation is called \textit{Riemannian} if its leaves are locally equidistant (see e.g. Molino \cite{molino1988riemannian} or Ghys \cite{ghys1984feuilletages}). 

The decomposition into the fibers of a Riemannian submersion is a main example of a Riemannian foliation: a submersion $\pi:M\to B$ is \emph{Riemannian} if the restriction $d\pi_p|_{(\ker d\pi_p)^\perp}$ is an isometry to $ T_{\pi(p)}B$ for every $p\in M$ (see e.g. O'Neill \cite{oneill} or Gromoll--Walschap \cite{gw}). The classification of Riemannian submersions from compact Lie groups with bi-invariant metrics was asked by Grove \cite[Problem 5.4]{grove2002geometry}.
% His question can be motivated in several ways. For instance, most examples of manifolds with positive sectional curvature are related to Riemannian submersions from Lie groups (a more complete account can be found in Ziller \cite{ziller2007examples}).

Indeed, in such groups all known Riemannian foliations with totally geodesic leaves are coset foliations.
Therefore, following Ranjan \cite{ranjan}, it is natural to ask whether coset foliations are the only Riemannian foliations with totally geodesic leaves on such groups. The affirmative answer to this question is supported by the following conjecture, commonly called ``Grove's Conjecture'' (see also Munteanu--Tapp \cite{tappmunteanu2}):

\begin{conj}\label{conj:grove}
Let $G$ be a compact simple Lie group with a bi-invariant metric. A Riemannian submersion  $\pi\co G\to B$ with connected totally geodesic fibers is induced either by left or right cosets.
\end{conj}

Here Ranjan's question together with Conjecture \ref{conj:grove} are proved affirmatively, without the simplicity assumption.

\begin{theorem}\label{thm:grove}
	Let $\pi \co G\to B$ be a Riemannian submersion with totally geodesic connected fibers on $G$,  a compact connected Lie group with bi-invariant metric. Then $\pi$ is isometric to a coset foliation. 
\end{theorem}

Actually, our proof has only one non-local instance, which can be circumvented by suitable compactness hypothesis. In particular, the proof works well for foliations on compact groups and reduce the general problem to foliations whose leaves are totally geodesic flats (as the foliation defined by the fibers of a vector bundle):

\begin{theorem}\label{thm:groveF}
	Let $\cal F$ be a Riemannian foliation with connected totally geodesic leaves on a connected open subset  $U$ of  a compact Lie group with bi-invariant metric $G$. Then $\cal V=T\cal F$ splits as $\cal V=\Delta_0\oplus\Delta_1$, where $\Delta_0$ defines a totally geodesic Riemannian foliation by Euclidean spaces and $\Delta_1$ is isometric to a coset foliation. Moreover, $\cal F$ is a coset foliation if it satisfies one of the following additional hypothesis:
	\begin{enumerate}[$(a)$]
		\item $\tilde L_\I$, the universal cover of a fixed leaf, has no Euclidean factors;
		\item $L_\I$ is complete and the integrability tensor of $\cal F$ is bounded along $L_\I$;
		\item the closure of $L_\I$ is a compact subset of $G$.
	\end{enumerate}
\end{theorem}

The proof is essentially algebraic, with two main geometrical instances: Theorems \ref{thm:A} and \ref{thm:tori}. Both are interesting on their own:  Theorem \ref{thm:A} is a refinement of the celebrated Ambrose--Singer Theorem for foliations with totally geodesic fibers on spaces with non-negative sectional curvature; Theorem \ref{thm:tori} could be readily used in an attempt to generalize Theorem \ref{thm:grove} to symmetric spaces.

The author is tempted to believe that the foliation defined by $\Delta_0$ is (locally) of a metric product $G=G'\times \bb R^k$, therefore it is a coset foliation.

The thesis of Theorem \ref{thm:grove} follows once we show that $\lie g$ decomposes in ideals $\lie g=\lie g_+\oplus \lie g_-$ such that, for every vectors $X,Y$ orthogonal to leaves,
\begin{equation}\label{eq:Aintroduction}
[X,Y]^v=\Big( [X_-,Y_-] - [X_+,Y_+]\Big)^v, 
\end{equation}
where $^v$ denotes the orthogonal projection to $\cal V=T\cal F$ and $X_\pm,Y_\pm$ are the $\lie g_\pm$-components of $X,Y$. It follows from \cite[Corollary 4.2]{tappmunteanu2} that $\pi$ is given by cosets (see section \ref{sec:1main} for details). The decomposition $\lie g_+\oplus\lie g_-$ is obtained by refining \cite[Theorem 1.5]{tappmunteanu2} and  ideas in \cite{ranjan}.

We observe that the hypothesis on Theorem \ref{thm:grove}  can not be relaxed:
Kerin--Shankar \cite{kerin-shankar} presented infinite families of Riemannian submersions  from compact Lie groups with bi-invariant metrics that can not be realized as principal bundles (for instance, the composition $h\circ pr:SO(16)\to S^8$ of the orthonormal frame bundle $pr:SO(16)\to S^{15}$ with the Hopf map $S^{15}\to S^8$ is one such submersion.) Moreover, the  simple group $SO(8)$  admits a foliation, $\cal F_{SO(8)}$, by totally geodesic round 7-spheres (obtained by trivializing the orthonormal frame bundle $SO(8)\to S^7$). Kerin--Shankar examples does not have totally geodesic fibers and $\cal F_{SO(8)}$ is not Riemannian. 

%One may wonder if such foliations are the only possible non-homogeneous (non-Riemannian) foliations with totally geodesic connected leaves on simple compact Lie groups.

The general classification of Riemannian foliations is wide open. For instance, classifications neither for totally geodesic Riemannian foliations on symmetric spaces, nor for generic  Riemannian foliations on Lie groups are known (we refer to Lytchak \cite{lytchak2014polar}, Lytchak--Wilking \cite{lytchak2016riemannian} and Wilking \cite{wilking2001index} for important developments in other cases). The author believes that the proof here can be partially replicated for the symmetric space case, giving important first steps.

\subsection{Preliminaries and description of each step}\label{sec:1main}
Given a Riemannian foliation $\cal F$ on $M$, we might  think of $\cal F$ locally as an stripped fabric (or a Riemannian submersion) with leaves vertically placed.  At each point $x\in M$, we decompose $T_xM$ as the tangent to the leaf $\cal V_x$ and its orthogonal complement $\cal H_x=(\cal V_x)^\bot$. We call $\cal V_x$ as the \textit{vertical space} and $\cal H_x$ as the \textit{horizontal space} at $x$. Given  $X\in TM$, we denote $X^h,X^v$ the horizontal and vertical components of $X$, respectively.  

A  vector field $X$ is said to be \textit{basic horizontal} if it is $\cal H$-valued and, for every vertical field $V$, $[X,V]$ is vertical. Equivalently, if $\cal F$ is induced by a submersion $\pi$, $X$ is basic if $d\pi(X)$ is fiberwise constant.  The flow of a basic horizontal vector field $X$ induces local diffeomorphisms between leaves (as a standard computation shows -- see e.g. Hirsch \cite[Proposition 17.6]{hirsch2012differential}). These (local) diffeomorphisms are called \textit{(local) holonomy transformations}. It is known that holonomy transformations are (local) isometries if and only if leaves are totally geodesic (see e.g. Gromoll--Walschap \cite[Lemma 1.4.3]{gw}), which is the case at hand.

Given a Riemannian foliation $\cal F$, the Gray--O'Neill integrability tensor $A\co \cal H\times\cal H\to \cal V$ is defined by 
\begin{equation*}
A_XY=\h[\bar X,\bar Y]^v,
\end{equation*}
where $\bar X,\bar Y$ are horizontal extensions of $X,Y$. We follow Ranjan \cite{ranjan} and denote $A^\xi X$ as the opposite dual of $A$ defined by:
\begin{equation*}
\lr{A^\xi X,Y}=-\lr{A_XY,\xi}.
\end{equation*}

Let $\phi_t$ be the flow  of a basic horizontal field $X$ and $c$ an integral curve of $\phi_t$. For any given $\xi\in\cal V_{c(0)}$, we define its \textit{holonomy field}  along $c$ by \[\xi(t)=d\phi_t(\xi).\]
Alternatively, $\xi(t)$ is the only vector field along $c$ that satisfies
\begin{gather*}
%\begin{cases}
\nabla_X \xi(t)=A^\xi X,\\ \xi(0)=\xi.\nonumber
%\end{cases}
\end{gather*}

The \textit{dual leaf at $p\in M$}, $L^\#_p$, is the subset of points in $M$ that can be joined to $p$ by horizontal curves (compare Wilking \cite{wilkilng-dual} or Gromoll--Walschap \cite[section 1.8]{gw}).

When the leaves of $\cal F$ are the fibers of a principal $G$-bundle $\pi\co P\to B$,  the integrability tensor, infinitesimal holonomy fields and dual leaves replace classical objects:  given a connection 1-form $\omega\co TP\to \lie g$, the curvature 2-form satisfies $\Omega(X,Y)=-2\omega(A_XY)$;  for any holonomy field along a  horizontal curve $c$, $\omega_{c(t)}(\xi(t))=\omega_{c(0)}(\xi(0))$. That is, $\xi(t)$ is the restriction to $c$ of the action field  defined by $\xi$; $L^\#_p$ is  \textit{the holonomy bundle  through $p$} (see e.g. \cite[secnomition II]{knI} for a definition of the last). The celebrated Ambrose--Singer Theorem \cite[Theorem 2]{ambrose-singer} identifies the Lie algebra of the holonomy group of $\pi$  with $\omega(T_pP(p))=\omega(L^\#_p)$. Theorem \ref{thm:A} refines this result in the case of Riemannian foliations/submersions on non-negatively curved ambient spaces:
\begin{theorem}\label{thm:A}
	Let $\cal F$ be a totally geodesic Riemannian foliation on a manifold $M$ of non-negative sectional curvature.  Let $L^\#_p$ be the dual leaf of $\cal F$ through $p$. Then
	\begin{equation*}
	TL^\#_p\cap \cal V_p=span\{A_XY~|~X,Y\in \cal H_p\}.
	\end{equation*}
\end{theorem}

Theorem \ref{thm:A} is used twice: to obtain a local version of the splitting theorem \cite[Corollary 3.3]{lytchak2014polar} and as one of the last steps in the paper.

When we are in the scope of Theorem  \ref{thm:groveF} (i.e., $\cal F$ is a totally geodesic Riemannian foliation on a neighborhood of  a Lie group with bi-invariant metric), Ranjan \cite{ranjan} makes a key observation: let $\I$ be the identity of $G$. For every $\xi\in\cal V_\I$, $X\in\cal H_{\I}$, the Grey--O'Neill's formulas imply: 
\begin{equation}\label{eq:Ranjan}
(A^\xi)^2=(\h\ad_\xi)^2.
\end{equation}
By using it, \cite{ranjan} proves Conjecture \ref{conj:grove} for simple groups that have a maximal torus inside a leaf. Such torus provides a decomposition of the basic horizontal fields, producing candidates for $\lie g_\pm$. Then the simplicity of the group is used to prove that either $\lie g_+$ or $\lie g_-$ is trivial.
Without the maximal torus assumption, we introduce  a new root system based on the integrability tensor of $\cal F$. 
\begin{theorem}\label{thm:tori}
	Let $\cal F$ be a Riemannian foliation with totally geodesic leaves on a manifold $M$. Let $L_p$ be the leaf through $p\in M$ and assume that a neighborhood of 0 in a subspace $\lie t^v\subseteq \cal V_p$ exponentiates to a totally geodesic flat. Suppose  that one of the hypothesis hold:
	\begin{enumerate}[$(a)$]
		\item $\exp(\lie t^v)$ is a complete totally geodesic flat and $A$ is bounded;
		\item $L_\I$ is the open neighborhood of a compact symmetric space without Euclidean factors.
	\end{enumerate}
Then, $R(\eta,\xi)^h=A^\eta A^\xi-A^\xi A^\eta$ for all $\xi,\eta\in \lie t^v$. 
\end{theorem}

Equation \eqref{eq:Ranjan} together with Theorem \ref{thm:tori} readily gives a decomposition $X=X_++X_-$, producing spaces $\cal H_+(\lie t^v)+\cal H_-(\lie t^v)\supseteq \cal H_\I$. The Lie algebras $\lie g_+,\lie g_-$ are the ideals generated by $\cal H_+(\lie t^v),\cal H_-(\lie t^v)$. The bulk of the paper is to prove that these ideals commute with each other. To this aid, we build
upon  Munteanu--Tapp \cite[Theorem 1.5]{tappmunteanu2} (Theorem \ref{thm:tapp} below), which provides Lie algebraic relations between the original root system and the one in Theorem \ref{thm:tori}:

A triple $\{X,V,\cal A\}\subseteq T_pM$ is called a \textit{good triple} if $\exp_p (tV(s))=\exp_p(sX(t))$ for all $s,t\in\bb R$, where $V(s),X(t)$ denote the Jacobi fields along $\exp(sX)$ and $\exp(tV)$ that satisfy $V(0)=V$, $X(0)=X$ and $V'(0)=\cal A=X'(0)$, respectively. Such conditions are achieved in  totally geodesic Riemannian foliations by $X\in\cal H$, $V\in \cal V$ and $\cal A=A^VX$. \cite[Theorem 1.5]{tappmunteanu2} provides a key identity that is used in section \ref{sec:good}:
\begin{theorem}[Theorem 1.5, \cite{tappmunteanu2}]\label{thm:tapp}
	Let $G$ be a compact Lie group with a bi-invariant metric and denote its Lie algebra by $\lie g$. The triple $\{X,V,A\}\subseteq \lie g$ is good if and only if, for all integers $n,m\geq 0$, 
	\begin{equation*}
	[\ad_X^nB,\ad_V^m\bar B]=0,
	\end{equation*}
	where $B=\frac{1}{2}\ad_VX-A$ and $\bar B= -\frac{1}{2}\ad_VX-A$.
\end{theorem}

Once proved that $\lie g_+$ commutes with $\lie g_-$, it follows from \cite[Proposition 4.1, Corollary 4.2]{tappmunteanu2} that the foliation is given by cosets: on the one hand
\begin{prop}[Munteanu--Tapp \cite{tappmunteanu2}, Proposition 4.1]\label{prop:MunteanuTapp}
Let $\cal F_1,\cal F_2$ be Riemannian foliations with totally geodesic leaves on $M$. Suppose that their vertical spaces, together with their integrability tensors coincide at a point $\I$. Then $\cal F_1=\cal F_2$. 
\end{prop}

On the other hand, assume that $M$ is a (open) subset of the product Lie group $G=G_1\times G_2$, equipped with a bi-invariant metric. Consider a subgroup $H<G_1\times G_2$ and define $\cal F_H$ as: 
\begin{equation}\label{eq:example_product}
L_{(g_1,g_2)}=\{(h_1g_1,g_2h_2^{-1})~|~(h_1,h_2)\in H \}.  
\end{equation}
Then,
\begin{equation}\label{eq:Asplit}
A^{(\xi_1,\xi_2)}(X_1,X_2)= (\h \ad_{\xi_1} X_1, -\h\ad_{\xi_2}X_2 ).
\end{equation}
Observe that inverting  the second coordinate on $G_1\times G_2$ is an isometry and interchanges $\mathcal F$ to a Riemannian foliation whose $A$-tensor satisfy $A^\xi X=\h\ad_\xi X$, for all $\xi\in \cal V_\I,X\in\cal H_\I$.

\begin{cor}[Munteanu--Tapp \cite{tappmunteanu2}, Corollary 4.2]\label{cor:TappMunteanu} 
	Let $\mathcal F$ be a Riemannian foliation with totally geodesic leaves on a connected Lie group $G$ with bi-invariant metric. If $A^\xi X=\h\ad_\xi X$, for all $\xi\in \cal V_\I,X\in\cal H_\I$, then $\cal V_\I$ is a subalgebra and $\cal F$ is the foliation defined by the left cosets of the subgroup  whose subalgebra is $\cal V_\I$.
\end{cor}

Although \cite[Proposition 4.1]{tappmunteanu2} assumes completeness of $M$, its proof can be carried out as far as every two points of $M$ can be joined by the concatenation of vertical and horizontal geodesics.
Theorem \ref{thm:groveF} follows by putting together \eqref{eq:Aintroduction}, \eqref{eq:Asplit} and Corollary \eqref{cor:TappMunteanu}. 

The paper is divided as follow: in section \ref{ap:AS} we prove Theorem \ref{thm:A} and use it to reduce Theorem \ref{thm:groveF} to the case of an irreducible foliation. Section \ref{sec:thmB} deals with the proof of Theorem \ref{thm:tori} and presents the splitting $\cal V=\Delta_0\oplus\Delta_1$. The proof of Theorems \ref{thm:grove}, \ref{thm:groveF} are completed in section \ref{sec:good}.

\mbox{}

The author would like to thank C. Dur\'an, K. Shankar and K. Tapp for suggestions and insightful conversations. Specially K. Shankar for pointing out \cite{berestovskii-nikonorov} (which was a crucial reference for an earlier version of the paper). The author also would like to thank Miguel Dom\'ingues V\'azquez and the anonymous referee for many suggestions, and Universidade Federal do Paran\'a for hosting the author for most part of this work. 
%This work was partially supported by CNPq grant number 404266/2016-9 and FAPESP grant number 2017/10892-7.

\section{An Ambrose-Singer theorem for non-negatively curved foliations}\label{ap:AS} 

Let $\pi\co M\to B$ be a Riemannian submersion. For simplicity we assume that all submersions and foliations here have totally geodesic fibers/leaves. Define 
\begin{equation*}\label{eq:dualleaf}
L^\#_p=\{c(1)\in M~|~c\co [0,1]\to M~\text{horizontal},~c(0)=p\}.
\end{equation*} 
We recall that, given a curve $\tilde c\co [0,1]\to B$, its horizontal lifts define the \textit{holonomy diffeomorphism} $\phi_{\tilde c}\co \pi^{-1}(\tilde c(0))\to \pi^{-1}(\tilde c(1))$ (by sending a point $q\in \pi^{-1}(\tilde c(0))$ to the endpoint of the horizontal lift of $\tilde c$ starting at $q$). When $\cal F$ is a Riemannian foliation, one can still define local diffeomorphisms, since $\cal F$ is locally given by submersions, therefore their differentials define holonomy fields. Thus, given a horizontal curve $c\co [0,1]\to M$ and $\xi\in \cal V_{c(0)}$, we denote $d\phi_c(\xi)=\xi(1)$, where $\xi(t)$ is the holonomy field defined by $\xi$, i.e., it satisfies:
\begin{gather}\label{key}
%\begin{cases}
\nabla_{\dot c} \xi(t)=A^\xi \dot c,\\ \xi(0)=\xi.\nonumber
%\end{cases}
\end{gather}

When $\pi$ is a principal $H$-bundle $L^\#_p\cap \pi^{-1}(b)$ coincides with an orbit of the holonomy group of $\pi$ at $b$. In this case, the Ambrose--Singer Theorem \cite{ambrose-singer} characterizes the Lie algebra of the holonomy group through the connection 2-form $\Omega$.
The result naturally extends to the case of a (not necessarily principal) Riemannian foliation (as one can see from the proof in \cite[section 3.4.2]{clarke2012holonomy}):

\begin{theorem}[Ambrose--Singer \cite{ambrose-singer}, Theorem 2]
	\label{thm:AS}
	Let $\cal F$ be a Riemannian foliation and $L_p^\#$  the dual leaf at $p\in M$. Then,
	%	Denote $\lie a_q=span\{A_XY~|~X,Y\in\cal H_q\}$. Then, for every  $p\in M$,
	\begin{equation*}
	T_p L_{p}^\#\cap \cal V_p=span\{d\phi_c^{-1}(A_XY)~|~c \text{ horizontal},~c(0)=p,~X,Y\in\cal H \}.
	\end{equation*}
\end{theorem}	

Although the Theorem gives a semi-local characterization for $TL_p^\#$, one must understand the behavior of the holonomy fields and of the $A$-tensor, which might be quite arbitrary objects. The situation can be greatly improved when the ambient space has non-negative sectional curvature. 

\begingroup
\def\thetheorem{\ref{thm:A}}
\begin{theorem}
	Let $\cal F$ be a totally geodesic Riemannian foliation on a manifold $M$ of non-negative sectional curvature.  Let $L^\#_p$ be the dual leaf of $\cal F$ at $p$. Then
	\begin{equation*}
	TL^\#_p\cap \cal V_p=span\{A_XY~|~X,Y\in \cal H_p\}.
	\end{equation*}
\end{theorem}
\addtocounter{theorem}{-1}
\endgroup

Under such hypothesis,  a vector in the cokernel of $A_X$ never leaves it (see Lemma \ref{prop:A-flatgeo}).  Observe that the result is absolutely local since there is no hypothesis on the completeness of $M$. Moreover, punctual information of $A$ spreads out through $M$: for instance, by combining Theorems \ref{thm:AS} and \ref{thm:A}, one concludes that  $A=0$ at  $p$ if and only if  $L^\#_p$ is a \textit{polar section} for $\cal F$ (i.e., $L^\#_p$ intersects every leaf perpendicularly). In this case, one can show that the universal cover $\tilde M$ metrically splits as $\tilde M=\tilde L^\#_p\times \tilde L_p$ (see  \cite[Theorem 1.4.1]{gw}).

Now we proceed to the proof of Theorem \ref{thm:A}. As in the introduction, denote $A^\xi\co \cal H_q\to \cal H_q$ as 
\[\lr{A^\xi X,Z}=-\lr{A_XZ,\xi}. \]
Denote $(\nabla_X A)^\xi Z=\nabla_X(A^\xi Z)-A^{\nabla_X^v\xi}Z-A^\xi \nabla_X^h Z$. By extending $\xi$  as a holonomy field,  one sees that:
\[\lr{(\nabla_X A)_XZ,\xi}=-\lr{(\nabla_X A)^\xi X,Z}.\qedhere\] 

In the remaining of the section, we assume $\cal F$ with totally geodesic leaves and $M$ with non-negative sectional curvature. Theorem \ref{thm:A} is based on the next inequality.

\begin{lem}\label{lem:WNN}
	%	Suppose that $\pi$ has totally geodesic fibers and $M$ nonnegative sectional curvature. Then, f
	For each $p\in M$ there is a neighborhood $U$ and a constant $a>0$ such that
	\begin{gather}\label{eq:nn}
	a\|X\|\|Z\|\|A^\xi X\|\geq |\lr{(\nabla_XA)^\xi X,Z}|
	\end{gather}
	for all $X,Z\in\cal H_q$ and $\xi\in\cal V_q$, $q\in U$.\end{lem}
\begin{proof}
	Given $X,Z\in\cal H$ and $\xi\in\cal V$,   Gray--O'Neill equations (\cite[page 44]{gw}) states  that the unreduced sectional curvature $K(X,\xi+tZ)=R(X,\xi+tZ,\xi+tZ,X)$ satisfies 
	\begin{equation}\label{eq:quadratictrick}
	K(X,\xi+tZ)=t^2K(X,Z)+2t\lr{(\nabla_XA)_XZ,\xi}+\|A^\xi X\|^2.
	\end{equation}
	Since  $K(X,\xi+tZ)\geq 0$, the discriminant of the polynomial \eqref{eq:quadratictrick} satisfies
	\begin{equation*}
	0\leq K(X,Z)\|A^\xi X\|^2-\lr{(\nabla_XA)_XZ,\xi}^2.
	\end{equation*}
	On small neighborhoods, continuity of $K$ guarantees the existence of some $a>0$ such that $K(X,Z)\leq a\|X\|^2\|Z\|^2$. 
\end{proof}

\begin{lem}\label{prop:A-flatgeo}
	%Suppose $\cal F$ has totally geodesic leaves and $M$ has nonnegative sectional curvature. 
	Let $\xi(t)$ be a holonomy field along  $\gamma(t)=\exp(tX)$, $X\in\cal H$. If $A^{\xi(0)}{X}=0$ then $A^{\xi(t)}{\dot \gamma(t)}=0$ for all $t$.
\end{lem}
\begin{proof}
	Take $\|X\|=1$ and  $Z=A^{\xi}{\dot \gamma}$ in \eqref{eq:nn}. Recalling that $\nabla_{\dot \gamma}^v\xi=0$, we get
	\begin{align}\label{proof:flatgeo1}
	a\|A^\xi{\dot \gamma}\|^2&\geq \lr{(\nabla_{\dot \gamma}A)^\xi{\dot \gamma},A^\xi {\dot \gamma}}=\lr{\nabla_{\dot \gamma}(A^\xi{\dot \gamma}),A^\xi {\dot \gamma}}=\frac{1}{2}\frac{d}{dt}\|A^\xi {\dot \gamma}\|^2.
	\end{align}
	Equation \eqref{proof:flatgeo1} is  Gronwall's inequality for $u(t)=\|A^{\xi(t)}{\dot c(t)}\|^2$, implying that
	\begin{equation*}\label{proof:flatgeo}
	\|A^\xi(t){\dot \gamma(t)}\|^2\leq \|A^{\xi(0)}{\dot \gamma(0)}\|^2e^{2 a t}
	\end{equation*}
	for all $t>0$. In particular, if $A^{\xi(0)}{\dot \gamma(0)}=0$, $A^{\xi(t)}{\dot \gamma(t)}=0$. Analogously, $A^{\xi(t)}{\dot \gamma(t)}=0$ for $t<0$ by  replacing  $X$ by $-X$ in the argument.
\end{proof}

The main ingredient in the proof is the constancy of the rank of  $\ker (A^{\xi(t)}\co\cal H_{c(t)}\to \cal H_{c(t)})$ along $\exp(\ker A^\xi)$ (Proposition \ref{prop:Dconstantrank}). First we present two  algebraic lemmas.

\begin{lem}\label{lem:in1}
	Let $X,Y\in\cal H$ be orthonormal and $A^{\xi}X=0$. Then,
	\begin{gather*}\label{eq:eigen0}
	{2 a}\|A^\xi Y\|^2\geq \lr{(\nabla_XA)^\xi Y+(\nabla_YA)^\xi X,A^\xi Y}.
	\end{gather*}
\end{lem}
\begin{proof}
	For any unitary $Z\in\cal H$, \eqref{eq:nn} gives:
	\begin{align*}
	2 a \|Z\|\|A^\xi {(Y+X)}\|&\geq \lr{(\nabla_XA)^\xi X+(\nabla_XA)^\xi Y+(\nabla_YA)^\xi X+(\nabla_YA)^\xi Y,Z},\\
	2 a \|Z\|\|A^\xi{(Y-X)}\|&\geq -\lr{(\nabla_XA)^\xi X-(\nabla_XA)^\xi Y-(\nabla_YA)^\xi X+(\nabla_YA)^\xi Y,Z}.
	\end{align*}
	Observe that $A^\xi{(X+Y)}=A^\xi{(Y-X)}=A^\xi Y$ and take $Z=A^\xi Y$. The result now follows by summing up both inequalities.
\end{proof}

Consider the non-negative symmetric operator $D=-A^\xi A^\xi$ and recall that $\ker A^\xi=\ker D$. Given $DY=\lambda^2 Y$, $\lambda >0$, define $\bar Y=\lambda ^{-1}A^\xi Y$. We have $\|\bar Y\|=\|Y\|$ and $A^\xi \bar Y=-\lambda Y$. In particular, if $\|Y\|=1$, $\|DY\|=\lambda^2$ and $\|A^\xi Y\|=\|A^\xi \bar Y\|=\lambda$.

\begin{lem}\label{lem:in2}
	Let $X,Y$ be unitary horizontals satisfying $A^\xi X=0$ and $DY=\lambda^2 Y\neq 0$. Then,
	\begin{equation*}
	\lr{(\nabla_YA)^\xi X,A^\xi Y}+\lr{(\nabla_{\bar Y}A)^\xi X,A^\xi \bar Y}=\lr{(\nabla_XA)^\xi \bar Y ,A^\xi \bar Y}.
	\end{equation*}
\end{lem}
\begin{proof}
	By combining the Bianch identity of $R^v(X,Y)Z$ and Grey--O'Neill's equation $R^v(X,Y)Z=-(\nabla^v_ZA)_XY$ we get (see also Lemma 1.5.1 in \cite[page 26]{gw}): 
	\begin{equation*}
	\lr{(\nabla_YA)^\xi X,\bar Y}=-\lr{(\nabla_XA)^\xi \bar Y,Y}-\lr{(\nabla_{\bar Y}A)^\xi Y,X}.
	\end{equation*}
	By replacing $A^\xi Y=\lambda\bar Y$ and $A^\xi\bar Y=-\lambda Y$, we have: 
	\begin{align*}
	\lr{(\nabla_YA)^\xi X,A^\xi Y}=&\lambda\lr{(\nabla_YA)^\xi X,\bar Y}=-\lambda[\lr{(\nabla_XA)^\xi \bar Y,Y}+\lr{(\nabla_{\bar Y}A)^\xi Y,X}]\\
	=&\lr{(\nabla_XA)^\xi \bar Y,A^\xi\bar Y}-\lr{(\nabla_{\bar Y}A)^\xi X,A^\xi\bar Y}.\qedhere  
	\end{align*}
\end{proof}

\begin{prop}\label{prop:Dconstantrank}  Let $\xi(t)$ be a holonomy field along $\gamma(t)=\exp(t X_0)$, $X_0\in \cal H_p$, and suppose that  $A^{\xi(0)} X_0=0$. If $\lambda(t)^2$ is a continuous eigenvalue of $D=-A^{\xi(t)}A^{\xi(t)}$, then $\lambda(t)$ either vanishes identically  or it never vanishes. 
\end{prop}
\begin{proof}
	We argue by contradiction. Assume that $\lambda$ vanishes at $t=0$  but there is $l>0$ such that $\lambda(t)>0$ for all $t\in (0,l)$. We further assume (by possibly reducing $l$) that $D$ has a smooth frame of eigenvectors along $\gamma((0,l))$.
	%We further assume (by possibly shrinking $l$) that $D$ admits a smooth frame of eigenvectors along $c((0,l))$. 
	The Proposition follows from Gronwall's inequality once we prove that
	\begin{equation}\label{eq:eigen}
	{ a'}\lambda^2\geq \frac{d}{dt}\lambda^2,
	\end{equation}
	for some $a'>0$.
	In particular,  $\lambda(t)^2\leq \lambda(\epsilon)^2e^{ a' t}$ for all $\epsilon\in (0,l)$, $t\in (\epsilon,l)$. Thus, $\lambda$ must vanish on $(0,l)$, a contradiction.
	
	Inequality \eqref{eq:eigen} follows from Lemmas \ref{lem:in1} and \ref{lem:in2}: let $Y$ be a smooth unitary vector field satisfying $DY=\lambda^2Y$. Applying Lemma \ref{lem:in1} on both $Y$ and $\bar Y$ gives
	\begin{align*}
	2 a\lambda^2&\geq \lr{(\nabla_XA)^\xi Y+(\nabla_YA)^\xi X,A^\xi Y},\\
	2{ a}\lambda^2&\geq \lr{(\nabla_XA)^\xi \bar Y+(\nabla_{\bar Y}A)^\xi X,A^\xi \bar Y}.
	\end{align*}
	Summing up gives:
	\begin{align*}
	4{ a}\lambda^2 &\geq \lr{(\nabla_XA)^\xi Y,A^\xi Y}+\lr{(\nabla_XA)^\xi \bar Y,A^\xi \bar Y}\\&+\lr{(\nabla_YA)^\xi X,A^\xi Y}+\lr{(\nabla_{\bar Y}A)^\xi X,A^\xi \bar Y}.\nonumber 
	\end{align*}
	Applying Lemma \ref{lem:in2}, we have 
	$\lr{(\nabla_YA)^\xi X,A^\xi Y}+\lr{(\nabla_{\bar Y}A)^\xi X,A^\xi \bar Y}=\lr{(\nabla_XA)^\xi \bar Y ,A^\xi \bar Y}$. On the other hand,
	\begin{align*}\label{proof:eigen1}
	\lr{(\nabla_XA)^\xi  \bar Y ,A^\xi  \bar Y} =&\lr{\nabla_X(A^\xi  \bar Y) ,A^\xi  \bar Y}-\lr{A^\xi(\nabla_X \bar Y),A^\xi  \bar Y} \\=&\h\frac{d}{dt}\lambda^2-\lambda^2\lr{\nabla_X \bar Y, \bar Y}=\h\frac{d}{dt}\lambda^2.
	\end{align*}
	Analogously, $\lr{(\nabla_XA)^\xi   Y ,A^\xi   Y}=\h\frac{d}{dt}\lambda^2$. Thus we can take $a'=\frac{8}{3}a$, concluding the proof.
\end{proof}

As the last step, fix $p\in M$ and denote  $\lie a_q=span\{A_XY~|~X,Y\in\cal H_q\}$. Observe that $\lie a_q^\bot=\{\xi\in\cal V_q~|~A^\xi =0\}$ and recall that a horizontal curve can be smoothly approximated by a broken horizontal geodesic. Then, Proposition \ref{prop:Dconstantrank} gives:

\begin{cor}\label{claim:2}
	For any  curve  $c$, $d\phi_c(\lie a_p^\bot)=\lie a_{c(1)}^\bot$.
\end{cor}
%\begin{proof}
%	It is sufficient to prove the claim for geodesics, since  any horizontal curve $c$ can be smoothly approximated by piece-wise horizontal geodesic $\gamma$. Let $\xi(t)$ be the holonomy field define by $\xi$ and recall that $d\phi_c(\xi)=\xi(1)$. Let $\xi\in\lie a_p^\perp$.
%	Since $A^\xi \cal H=0$, In particular, if $\xi\in\lie a_p^\bot$, and $\ker A^\xi(t)$ has constant rank (Lemma \ref{prop:Dconstantrank}). Thus, $\xi(1)  \in \lie a^\bot_{\gamma(1)}$.
%\end{proof}

Theorem \ref{thm:A} follows directly from Corollary \ref{claim:2}: since $d\phi_c$ is an isometry, Lemma \ref{claim:2} implies that $d\phi_c^{-1}(\lie a_{\phi_c(p)})=\lie a_p$. Applying Theorem \ref{thm:AS}, one directly gets the equality $TL^\#_p\cap \cal V_p=\lie a_p$. \halmos

Once established Theorem \ref{thm:A}, one can observe that $\lie a_p^\perp$ coincides with $\nu(L^\#_p)$, the space normal to $L^\#_p$. It gives the very important local version of \cite[Proposition 3.1]{lytchak2014polar} below. It guarantees the same thesis of \cite[Proposition 3.1]{lytchak2014polar} by exchanging the completeness of $M$ by the assumption of totally geodesic leaves in $\cal F$.

\begin{cor}\label{cor:flat}
	Let $\cal F$ be a totally geodesic Riemannian foliation on $M$, a non-negatively curved manifold. 
	%	Let $\gamma$ be an $\cal F$-horizontal geodesic starting at  $p\in M$. 
	Then the sectional curvature $sec(\xi,X)=0$, for every $\xi\in \nu(L^\#)$ and $X\in \cal H$ and  $\nu(L^\#)$ is parallel translated along $\exp(tX)$.
\end{cor}
\begin{proof}
	Along the proof of Theorem \ref{thm:A}, we have shown that the distribution $p\mapsto \lie a^\perp_p=\nu(L^\#_p)$ is invariant along holonomy transformation defined by horizontal geodesics. 
	Moreover, if  $\xi(t)$ is the holonomy field defined by $\xi\in\lie a^\perp _p$ along $\gamma(t)=\exp(tX)$, $X\in \cal H_p$, then \[sec(\xi(t),\dot \gamma(t))=\frac{\|A^{\xi(t)} \dot \gamma(t)\|^2}{\|\dot \gamma(t)\|^2\|\xi(t)\|^2}\equiv0.\]
	On the one hand,  $\nabla_{\dot \gamma(t)}\xi(t)=0$, thus $\xi(t)$ is parallel. On the other hand, $\xi(t)\in\lie a^\perp_{\gamma(t)}$ for every $t$, therefore $\lie a^\perp$ is parallel along $\gamma$.	
\end{proof}

\subsection{Reduction to  the single-dual-leaf case}\label{sec:pi1}

Here we reduce the proof of Theorem \ref{thm:groveF} to the case of an irreducible $\cal F$ (i.e., with only one dual leaf). 
First we observe that it is sufficient to prove Theorem \ref{thm:groveF} locally: if $\{U_i\}$ is an open cover of $U$ such that $\cal F|_{U_i}$ is the restriction of a coset foliation defined by a connected subgroup, then $\cal F|_{U_i}$ and $\cal F|_{U_j}$ must be the restriction of the same coset foliation whenever $U_i\cap U_j\neq \emptyset$. 
To reduce to the case of only one dual leaf, we give a local version of the following result due to Lytchak (see also \cite{silva2020completeness}):

\begin{theorem}[Corollary 3.3 \cite{lytchak2014polar}]\label{thm:lytchak}
	Let $\cal F$ be a regular Riemannian foliation on a simply connected compact symmetric space $M$. Then there is a metric decomposition  $M=M_1\times M_2$ and a foliation $\cal F_1$ on $M_1$ such that each slice $M_1\times\{x_2\}$ is a dual leaf for $\cal F$ and $\cal F$ satisfies
	\[\cal F=\{L\times M_2~|~L\in \cal F_1\}. \]
\end{theorem}

Again, let $\cal F$ be a Riemannian submersion and consider the decomposition $G=G_1\times G_2$ given by Theorem \ref{thm:lytchak}. If  $\cal F_1$ is given by left $H$-cosets, $H<G_1$, then $\cal F$ is given by the left $H\times G_2$-cosets. 

We now establish our local version of Theorem \ref{thm:lytchak}:

\begin{prop}\label{prop:Lytchak}
	Let $\cal F$ be a totally geodesic  Riemannian foliation on a simply connected symmetric space $M$ with non-negative sectional curvature. Then there is a metric decomposition  $M=M_1\times M_2$ and a foliation $\cal F_1$ on $M_1$ such that each slice $M_1\times\{x_2\}$ is a dual leaf for $\cal F$ and $\cal F$ satisfies
\[\cal F=\{L\times M_2~|~L\in \cal F_1\}. \]	
\end{prop}
\begin{proof}
Theorem \ref{thm:lytchak} is based on \cite[Proposition 3.1]{lytchak2014polar} and the completeness of dual leaves of regular Riemannian foliations on compact non-negatively curved manifolds (\cite[Theorem 3, item (b)]{wilkilng-dual}). Here we argue on how to trade completeness by the assumption of totally geodesic leaves on both points.

We first observe that dual leaves of (non-singular) totally geodesic foliations must be complete. In such case, holonomy transformations are local isometries of leaves, and can be lifted as full isometries defined on the whole universal cover o each leaf. Therefore the intersection of the dual leaf to a leaf $L$, $L^\#\cap L$ is (finitely covered by) the orbit of a proper group action on $\tilde L$. It follows that $L^\#$ is complete since $L^\#$ is invariant under holonomy transformations.

On the other hand, the proof of Proposition 3.1 in \cite{lytchak2014polar} is based on simple Lie algebraic computations (by identifying $M=G/K$ and $\lie g=TM\oplus \lie k$) and the following two properties: 
$sec(\nu(L^\#),\cal H)=0$ and that $\nu(L^\#)$ is invariant with respect to parallel translations along horizontal geodesics. These two facts are guaranteed by the arguments in \cite{wilkilng-dual} for singular Riemannian foliation on a complete ambient space  with non-negative sectional curvature. In our case, these two facts were recovered by Corollary \ref{cor:flat}.
\end{proof}

Considering the arguments above, it is sufficient to prove Theorems \ref{thm:grove}, \ref{thm:groveF} assuming that $\cal F$ has only one dual leaf.
This assumption is used in section \ref{sec:proof} in order to apply Theorem \ref{thm:A}.

\section{The $A$-root system}\label{sec:thmB}

Both in here and in section \ref{sec:good}  we work with the complexification of some related spaces, specially $\lie g$ and $\cal H_\I$. Given a vector space $V$, the complexification of $V$ will be denoted by $V^{\bb C}$. Given an operator $A\co V\to V$, its natural complexification is denoted by the same letter $A\co V^\bb C\to V^\bb C$ and is defined by $A(x+iy)=A(x)+iA(y)$. We follow Knapp \cite{knapp2013lie} and extend inner products on $V$ to $\bb C$-bilinear symmetric products on $V^\bb C$ (not to Hermitian, positive definite ones).

The usual setting for a root system  consists of an abelian real Lie algebra $\lie t$ acting on a linear space $V$ through a Lie algebra morphism $\rho\co\lie t\to \End(V)$. For instance, one may endow $V$ with an inner product and suppose that $\rho(\lie t)\subseteq \End(V)$ is a subspace of commuting skew-adjoint linear endomorphisms of $V$. In this case, $\rho$ naturally defines an action on $V^{\bb C}$ and the  subset $\rho(\lie t)\subseteq \End(V^\bb C) $ consists of endomorphisms with pure imaginary eigenvalues that can be diagonalized in a single bases. The \textit{root decomposition induced by $\rho(\lie t)$} is defined by 
\begin{equation*}
V^{\bb C}=\sum_{\alpha\in\Pi} V_\alpha.
\end{equation*}
where $V_\alpha$ is the \textit{weight space} of the linear function $\alpha:\lie t\to i\bb R$: 
\[V_\alpha=\{X\in V^{\bb C}~|~\rho(A)X=\alpha(A)X,~\forall A\in\lie t\}.\]
Whenever $V_\alpha\neq \{0\}$, for $\alpha\neq 0$, we call $\alpha$ a \textit{root}. We denote the set of roots by $\Pi(\lie t)$.

Let $\cal F$ be a Riemannian foliation. Let $\imath \co \lie t^v \looparrowright L_p$ be an immersed  totally geodesic flat with $\imath(0)=p$. Here we prove  that, under certain conditions, $\rho_A(\xi)=A^\xi$ defines a representation of $\lie t^v$ on $\cal H_p$.
%the space of basic horizontal fields along $\imath^*\cal H$. 
We use the following convention for the Riemannian curvature:
\[R(X,Y)Z=\nabla_X\nabla_YZ-\nabla_Y\nabla_XZ-\nabla_{[X,Y]}Z. \]

\begingroup
\def\thetheorem{\ref{thm:tori}}
\begin{theorem}
	Let $\cal F$ be a Riemannian foliation with totally geodesic leaves on a manifold $M$. Let $L_p$ be the leaf through $p\in M$ and that a neighborhood of 0 in a subspace $\lie t^v\subseteq \cal V_p$ exponentiates to a totally geodesic flat. Suppose  that one of the hypothesis hold:
	\begin{enumerate}[$(a)$]
		\item $\exp(\lie t^v)$ is complete and $A$ is bounded;
		\item $L_\I$ is the open neighborhood of a compact symmetric space.
	\end{enumerate}
	Then, $R(\eta,\xi)^h=[A^\eta,A^\xi]$ for all $\xi,\eta\in \lie t^v$. 
\end{theorem}
\addtocounter{theorem}{-1}
\endgroup

We begin by proving item $(a)$.
\begin{proof}[Proof of item $(a)$]
Consider   basic horizontal fields $X,Y$  and vertical fields $\xi,\eta$ such that $\nabla_\xi\eta=\nabla_\eta\xi=0$ along $\exp(\lie t^v)$.  
%Extend  $\xi,\eta$ in the directions of $X,Y$ as holonomy fields, so that $\nabla_\xi X=A^\xi X$, $\nabla_\eta X=A^\eta X$. 
We have,
\begin{align*}
\lr{R(\eta,\xi)X,Y}=&\lr{\nabla_\eta\nabla_\xi X-\nabla_\xi\nabla_\eta X,Y} = \lr{\nabla_\eta (A^\xi X)- \nabla_\xi(A^\eta X),Y}\\  
 =& -\eta\lr{A_XY,\xi}+\xi\lr{A_XY,\eta}-\lr{ A^\xi X,A^\eta Y}+\lr{A^\eta X,A^\xi Y}\\
 =& -\lr{\nabla_\eta(A_XY),\xi}+\lr{\nabla_\xi(A_XY),\eta}+\lr{(A^\eta A^\xi-A^\xi A^\eta)X,Y}.
\end{align*}
The proof is concluded by observing that $R^h(\eta,\xi)\cal V_p=0$, since fibers are totally geodesic, and $\lr{\nabla_\xi(A_XY),\eta}=-\lr{\nabla_\eta(A_XY),\xi}=0$.
For the last, consider the geodesic $\gamma(s)=\exp(t\xi)$ and recall that $A_XY$ is a Jacobi field along $\gamma$. We have,
\begin{equation*}\label{proof:tori}
	\xi\xi\lr{A_XY,\eta}=\lr{\nabla_\xi\nabla_\xi(A_XY),\eta}=\lr{R(\xi,\eta)\xi,A_XY}=0.
	\end{equation*}
	Therefore,  $\varphi(t)=\lr{A_{X}Y,\eta}(\gamma(t))$  is a bounded affine function in the real line.  In particular, $\varphi(t)$ must be constant and $\lr{\nabla_\xi(A_XY),\eta}=\xi\lr{A_XY,\eta}$ vanishes.
\end{proof}

In the proof above, the boundedness of $A$ is used to show that $A_XY|_{\exp(t\xi)}$ is bounded. Fortunately, there is a natural way to ensure such a bound using only local information. 

Let $\cal F$ be as in Theorem \ref{thm:groveF} i.e., $\cal F$ is a totally geodesic Riemannian foliation on an open subset  $U\subseteq G$ of a Lie group $G$ with bi-invariant metric. We assume that $\I\in U$ without lost of generality. Recall that, although $L_\I$ is defined only on $U$, $L_\I$ can be isometrically identified with an open neighborhood of  a complete symmetric space $\tilde L_\I$ (\cite[Theorem 5.1]{helgasondifferential}).  Since $\tilde L_\I$ must have non-negative sectional curvature, it is locally isometric to a product $L_0\times L_1$, where $L_0$ is an Euclidean space and $L_1$ is a compact symmetric space. It is easy to conclude that $A_XY$ can only have unbounded components on $L_0$. With this motivation in mind, we prove item $(b)$

%\begin{prop}
%	Let $\cal F$ be a totally geodesic Riemannian foliation on an open subset  $U\subseteq G$ of a Lie group $G$ with bi-invariant metric. Then the thesis of Theorem \ref{thm:tori} holds, if $L_\I$ has no Euclidean factors.
%\end{prop}
\begin{proof}[Proof of  item $(b)$]
Observe that every isometry $\varphi:L_\I\to L_\I$ can be extended to an isometry of $\tilde L_\I$ (for instance, if $\phi\co L_\I\to L_\I$ is an isometry, then its graph  is a closed totally geodesic submanifold $\Gamma\subseteq L_\I\times L_\I$, thus a symmetric space by itself. Therefore, there is a unique symmetric space $\tilde \Gamma$ containing $\Gamma$ as an open subset. One clearly sees that $\tilde \Gamma$ can be naturally identified as a submanifold of $\tilde L_\I\times \tilde L_\I$ which is the graph of an isometry $\tilde\varphi$). In particular, Killing fields on $L_\I$ are the restriction of Killing fields in $\tilde L_\I$.

Since $\tilde L_\I$ is locally the product $L_0\times L_1$,  a Killing field $\zeta$ in $L_\I$ decomposes as the sum of a component $\zeta_0$ in $L_0$ and $\zeta_1$ in $L_1$. Since $L_1$ is compact, $\zeta$ is unbounded only if $\zeta_0$ is unbounded. The result now follows since $A_XY|_{L_\I}$ is a Killing field, whenever $X,Y$ are basic horizontal.
\end{proof}

\subsection{Splitting of totally geodesic foliations}\label{sec:split}

At last, we observe that, even if $L_\I$ has an Euclidean factor, we can still split the foliation and use Theorem \ref{thm:tori} on the compact factor. The splitting we mean is stated in the next result, which should be either known or expected to hold among specialists.

\begin{theorem}\label{thm:split}
	Let $\cal F$ be a totally geodesic Riemannian foliation  with simply connected dual leaves. For a fixed $L\in \cal F$, let $TL=\bigoplus_{i=0}^s\tilde\Delta_i$ be the de Rham decomposition of $TL$. Then, there are smooth integrable distributions $\Delta_0,...,\Delta_s$ on $M$ such that, for every $i$:
	\begin{enumerate}[$(1)$]
		\item $\Delta_i$ is vertical and $\cal V=\bigoplus_{i=0}^s\Delta_i$;
		\item for every leaf $L'\in \cal F$, $TL'=\bigoplus_{i=0}^s\Delta_i|_{L'}$ is the de Rham decomposition of $TL'$;
		\item $\cal F_i$, the foliation defined by $\Delta_i$, is Riemannian and has totally geodesic leaves.
	\end{enumerate}
\end{theorem}

Let $\cal F$ be an irreducible Riemannian foliation with totally geodesic leaves. Let $L_p$ be the leaf through $p\in M$ and denote $TL_p=\bigoplus_i\tilde\Delta_i$ as the de Rham decomposition of $TL_p$. 

Let $c\co[0,1]\to M$ be a horizontal curve. Recall that holonomy fields define a linear isometry $d\phi_c\co \cal V_{c(0)}\to \cal V_{c(1)}$, $d\phi_c(\xi)(\xi(1))$, where $\xi(t)$ is the holonomy field defined by $\xi$ along $c$. Moreover,  by recalling that Riemannian foliations are locally given by Riemannian submersions,  one concludes that $c$ defines an isometry between universal covers $\phi_c\co \tilde L_{c(0)}\to \tilde L_{c(1)}$.

Fix a leaf $L\in \cal F$ and $TL=\bigoplus \tilde \Delta_i$, its de Rham decomposition. Define $\Delta_i(q)=d\phi_c(\tilde{\Delta}_i)$, where $c$ is a horizontal curve joining $p$ to $q$. We claim that $\Delta_i$ is well defined if dual leaves are simply connected.

Let $c_1,c_2:[0,1]\to M$ be horizontal curves joining $L$ to $q$. Then $(d\phi_{c_2})^{-1}d\phi_{c_1}=d\phi_{c}$, where $c$ is the concatenation of $c_1$ with the reverse of $c_2$. In particular, $d\phi_{c_1}(\tilde\Delta_i)=d\phi_{c_2}(\tilde \Delta_i)$ for every pair of horizontal curves $c_1,c_2$, such that $c_1(0),c_2(0)\in L$ and $c_1(1)=c_2(1)$, if and only if  $d\phi_{c}(\tilde \Delta_i)=\tilde \Delta_i$ for every horizontal curve $c$, such that $c(0),c(1)\in L$.
Denote
\[\Hol_L=\{\phi_c\co\tilde L_p\to \tilde L_p~|~c(0),c(1)\in L,~c~\text{horiontal}\}.\]
Observe that $\Hol_L$ is a subgroup of isometries of $\tilde L_p$. Moreover, according to Eschenburg--Heintze \cite{Eschenburg-Heintze}, it is sufficient to show that $\Hol_L$ does not exchange factors of the de Rham decomposition of $\tilde L$.
	
On the one hand, if $\Hol_L$ exchange factors, the action of $\Hol_L$ on  $\tilde L$ has non-connected isotropy group at some point (if $\tilde L$ has two isometric factors $M_1\times M_1$, then the points in the diagonal have non-connected isotropy). On the other hand, if $\phi_c(p)=p$, then $c(0)=c(1)=p$ and $d\phi_c$ is naturally identified with the isotropy  representation of $\phi_c$ at $p$. Therefore, it is sufficient to prove that the group:
\[ H_p=\{d\phi_c\co \cal V_p\to \cal V_p~|~c(0)=c(1)=p,~c~\text{horiontal}\} \]
is connected. 

\begin{claim}
	The distribution $\tilde\Delta_i$ is $\cal H_p$-invariant.
\end{claim}
\begin{proof}
%	To prove so, we briefly recall (in a non-direct manner) the notion of \textit{infinitesimal holonomy groupoid} introduced in \cite{speranca2017on} (here presented as a suitable dual leaf).
Let $\bar{\pi}\co O(\cal V)\to M$ be the bundle of orthonormal frames of $\cal V$, i.e., 
\[O(\cal V)=\{b\co \bb R^k\to \cal V_p~|~b~\text{linear isometry}\}.\] 
$O(k)$ acts on $O(\cal V)$ by right composition. Observe that $\tilde{\cal F}=\{\bar{\pi}^{-1}(L)~|~ L\in \cal F \}$
defines a foliation on $O(\cal V)$. One can make $\tilde{\cal F}$ Riemannian by observing that $\nabla^\cal V_X\xi=(\nabla_X\xi)^v$ defines a $O(k)$-invariant $\tilde{\cal F}$-horizontal distribution $\tilde{\cal H}$ and equipping each fiber with a bi-invariant metric. 

Let $b\in O(\cal V)$ be such that $\bar \pi(b)=p$. Denote the $\tilde{\cal F}$-dual leaf at $b$ by $\cal E_b$. One can observe that $\bar \pi|_{\cal E_b}\co \cal E_b\to L^\#_p$ is a principal bundle with principal group $b^{-1}(\cal E_b\cap \bar{\pi}^{-1}(p))=b^{-1}\cal H_pb$. By definition, every point in $\cal E_b$ is connected to $b$ by a $\tilde{\cal F}$-horizontal curve, in particular, by a $\bar \pi$-horizontal curve (i.e., orthogonal to the $\bar\pi$-fibers). Therefore,  $\bar \pi|_{\cal E_b}$ is irreducible as  a principal bundle, concluding that $H$ is connected, since $L^\#_p$ is simply connected (see \cite{knI} for details).
%\end{proof}
\end{proof}

Observe that $\Delta_i\subseteq \cal V$ and $\bigoplus\Delta_i|_{L_q}$ is a de Rham decomposition in each leaf $L_q$. Therefore $\Delta_i|_{L_q}$ integrates a Riemannian foliation with totally geodesic leaves on each $L_q$. Since $L_q$ is totally geodesic on $M$, the integral submanifolds of $\Delta_i$ are totally geodesic on $M$. It is left to prove that the foliation defined by $\Delta_i$ on $M$ is Riemannian. 

\begin{claim}
	Let $\cal F_i$ be the foliation defined by $\Delta_i$. Then $\cal F_i$ is Riemannian.
\end{claim}
\begin{proof}
	We follow Gromoll--Walschap  \cite[Theorem 1.2.1]{gw} and show that the Lie derivative $\cal L_{U}g^{\Delta_i^\perp}=0$ for every $U\in\Delta_i$. Observe that $\Delta_i^\bot=\cal H\oplus (\oplus_{i\neq j}\Delta_j)$ and: $\cal L_Ug(\cal H, \cal V)=0$ and $\cal L_{U}g(\cal H, \cal H)=0$ since $\cal F$ is Riemannian and $U$ vertical; $\cal L_{U}g({\Delta_j},\Delta_k)=0$, $j,k\neq i$, since the leaves of $\cal F$ are totally geodesic and the restriction of $\Delta_i$ to each leaf is Riemannian. 
\end{proof}

\section{Good triples and the $\cal H_\pm$-decomposition}\label{sec:good}

From now on, we specialize to the case of a totally geodesic Riemannian foliation $\cal F$ on $U\subseteq G$,   a compact Lie group with bi-invariant metric. Furthermore, we assume that $\cal F$ satisfy the thesis hypothesis in Theorem \ref{thm:tori} and use it throughout. We call such a foliation as a \textit{Ranjan foliation}. 

This section is the technical bulk of the paper and concludes the proof of Theorem \ref{thm:grove}. Here we decompose $\lie g$ in commuting ideals $\lie g_+,\lie g_-$ satisfying \eqref{eq:Aintroduction} (we actually consider a third ideal $\lie g_0$ for technical reasons, but it can be incorporated in either $\lie g_+$ or $\lie g_-$.) The decomposition $\lie g_+,\lie g_-$ will be achieved step by step: in section \ref{sec:brac1} we fix a maximal abelian subalgebra inside $\cal V_\I$ and decompose the elements of $\cal H_{\I}$ according to the relation between the root system of $\lie g$ and of Theorem \ref{thm:tori}. The process produces subspaces  $\cal H_+(\lie t^v),\cal H_-(\lie t^v),\cal H_0(\lie t^v)\subseteq \lie g$; section \eqref{sec:brac2} proves a strong commuting identity for $\cal H_+(\lie t^v),\cal H_-(\lie t^v)$ which is used throughout; in section \ref{sec:brac3} we expand $\cal H_\pm(\lie t^v)$ to subspaces $\lie H_\pm(\cal F)$ which are independent of the choice of $\lie t^v$; in \ref{sec:proofHpm} we prove that $\cal H_\pm(\cal F)$ commute, providing the very important Lemma \ref{cor:A}. Using Lemma \ref{cor:A} and the irreducibility hypothesis, we put Theorem \ref{thm:A} into play in order to prove that $\ad_{\cal V_\I}$ preserves the subalgebras generated by $\mathcal H_\pm(\cal F)$; finally, in section \ref{sec:proof} we prove that the algebras generated by $\cal H_\pm(\cal F)$ are ideals and that they commute. Together with Lemma \ref{cor:A} and Proposition \ref{prop:MunteanuTapp}, it concludes the proof of Theorem \ref{thm:grove}.

The foliation defined by \eqref{eq:example_product} gives a picture of $\cal H_\pm(\cal F)$: let $H<G=G_1\times G_2$ and consider $\cal F_H$ as the foliation defined by the orbits of $(h_1,h_2)\cdot(g_1,g_2)=
(h_1g_1,g_2h_2^{-1})$.  Each  vector in $\cal H_{(g_1,g_2)}$ has a component tangent to $G_1\times\{g_2\}$ and other tangent to $\{g_1\}\times G_1$. The  subsets  spanned by such components are the desired subspaces $\cal H_+(\cal F)$ and $\cal H_-(\cal F)$, respectively. In this case the ideals generated by $\cal H_+(\cal F),\cal H_-(\cal F)$ clearly commute. The whole paper is dedicated to show that this is the general situation.

All arguments in section \ref{sec:brac1}-\ref{sec:proofHpm} follows from Theorem \ref{thm:tori} and Munteanu--Tapp's Theorem \ref{thm:tapp}. That is, assuming that both theorems hold at $\I$, the results in section \ref{sec:brac1}-\ref{sec:proofHpm} holds.  Section \ref{sec:proof} further requires Lytchak's decomposition Theorem \ref{thm:lytchak} (or its local version, Proposition \ref{prop:Lytchak}) to reduce the general case to the case of  a single dual leaf, which is required to apply Theorem \ref{thm:A}.

Given a Ranjan foliation, the leaf through $\I$, $L_\I$, is a totally geodesic submanifold of a symmetric space, thus a symmetric space itself. In particular, $\lie t^v=\lie t\cap \cal V_\I$ exponentiates to a maximal totally geodesic flat in  $L_\I$, as long as $\lie t$ is a totally geodesic abelian subalgebra in $\lie g$.

At last, we recall that the $(4,0)$ Riemannian curvature tensor of a bi-invariant metric satisfies:
\begin{equation}\label{eq:R}
R(X,Y,Z,W)=-\frac{1}{4}\lr{[X,Y],[Z,W]}.
\end{equation}

\subsection{The horizontal decomposition I}\label{sec:brac1}
Consider a maximal vertical abelian  subalgebra $\lie t^v\subseteq \cal V_{\I}$ completed to a maximal abelian subalgebra $\lie t=\lie t^v\oplus \lie t'\subseteq \lie g=T_\I G$. $\lie t$ and $\lie t^v$  act  on $\lie g$ through the representations  $\rho_{\ad}(\xi)=\frac{1}{2}\ad_\xi$ and  $\rho_A(\xi)=A^\xi$, respectively.

Given  linear maps $\alpha\co \lie t^v\to i\bb R$, $\alpha'\co\lie t'\to i\bb R$, we consider the spaces:
\begin{align*}
\lie g_{\alpha,\alpha'}(\lie t)&=\{X\in\lie g^{\bb C}~|~\textstyle{\h}\ad_{\xi+\xi'}X=(\alpha(\xi)+\alpha'(\xi'))X,~ \text{for all }\xi\in\lie t^v,~ \xi'\in\lie t' \},\\
\lie g_\alpha(\lie t^v) &=\{X\in\lie g^{\bb C} ~|~\textstyle{\h}\ad_{\xi}X=\alpha(\xi)X,~ \text{for all }\xi\in\lie t^v\},\\
\cal H_\alpha(\lie t^v) &=\{X\in\HH ~|~A^{\xi}X=\alpha(\xi)X,~ \text{for all }\xi\in\lie t^v \}.
\end{align*}
We call their elements as $(\alpha,\alpha')$-weights, vertical $\alpha$-weights and $\alpha$-$A$-weight, respectively. Whenever one of such spaces is non-trivial, the corresponding linear map is called a root, vertical root or $A$-root, respectively. We denote the corresponding set of roots as $\Pi(\lie t)$, $\Pi^v(\lie t^v)$ and $\Pi^A(\lie t^v)$.

%, we call $X\in\lie g^{\bb C}$ (respectively $X\in\HH$) a  \textit{vertical $\alpha$-weight} (respectively an \textit{$\alpha$-$A$-weight}) if, for all $\xi\in\lie t^v$, $\h \ad_\xi(X)=\alpha(\xi)X$ (respectively,  $A^\xi(X)=\alpha(\xi)X$). If $\alpha$ is such that there is  a non-trivial vertical $\alpha$-weight (respectively, $\alpha$-$A$-weight) then $\alpha$  is called a \textit{vertical root} (respectively, an \textit{$A$-root}). Denote the set of vertical roots as $\Pi^{v}(\lie t^v)$ (and the set of $A$-roots as $\Pi^{A}(\lie t^v)$). 

Observe that a vertical root $\alpha$ can always be completed to a root $(\alpha,\alpha')\co\lie t\to i\bb R$ of $\lie g$ by  a linear function $\alpha'\co\lie t'\to i\bb R$:  since $\lie t'$ commutes with $\lie t^v$, $\lie g_\alpha(\lie t^v)$ can be further decomposed as
\[\lie g_\alpha(\lie t^v)=\sum_{\alpha'}\lie g_{\alpha,\alpha'}(\lie t). \]
Conversely, the identification $(\lie t\oplus\lie t')^*=\lie t^*\oplus (\lie t')^*$ is given by the restriction $\tilde \alpha\mapsto (\tilde\alpha|_{\lie t^v},\tilde \alpha|_{\lie t'})$. In particular, if $(\alpha,\alpha')$ is a root, $\alpha$ is a vertical root. 
We take advantage of this two-level decomposition:
\begin{equation*}
\lie g^{\bb C}=(\lie t^v)^{\bb C}+\sum_{\alpha\in\Pi^v(\lie t^v)}\lie g_\alpha(\lie t^v)=\lie t^\bb C+\sum_{(\alpha,\alpha')\in \Pi(\lie t)}\lie g_{\alpha,\alpha'}(\lie t),
\end{equation*}
and the decomposition based on Theorem \ref{thm:tori}:
\[\HH=\cal H_0(\lie t^v)+\sum_{\alpha\in\Pi^A(\lie t^v)}\cal H_\alpha(\lie t^v), \]
where $\cal H_0(\lie t^v)=\cap_{\xi \in {\lie t^v}}\ker A^\xi$. The $\lie g_{\alpha,\beta}(\lie t)$-, $\lie g_\alpha(\lie t^v)$-, $\cal H_\alpha(\lie t^v)$-components of $X$ will be denoted by $X_{\alpha,\beta}$, $X_\alpha$, $X^\alpha$, respectively.

Our first step in this algebraic part is to relate $A$-weights to vertical weights.

\begin{lem}\label{lem:root1}
	Let $\lie t^v$ be a maximal vertical abelian subalgebra. Then, 
\begin{equation*}\label{eq:PiA}
	\Pi^{A}(\lie t^v)=\{\alpha\in \Pi^{v}(\lie t^v)~|~\cal H_e^\bb C\cap (\lie g_\alpha(\lie t^v)+\lie g_{-\alpha}(\lie t^v))\neq \emptyset\}\subseteq \Pi^v(\lie t^v).
\end{equation*}
	Moreover, if $X\in\HH$ is an $\alpha$-$A$-weight, then 
	\begin{equation*}\label{eq:Xalpha}
	X=X_{\alpha}+X_{-\alpha}\in \lie g_\alpha(\lie t^v)+\lie g_{-\alpha}(\lie t^v).
	\end{equation*}
%where $X_{\pm}\in\lie g_{\pm\alpha}(\lie t^v)$ is the $\lie g_{\pm\alpha}(\lie t^v)$-component of $X$.
\end{lem}
\begin{proof}
	Since we are dealing with Ranjan foliations, Grey--O'Neill's equations gives  (compare Ranjan \cite{ranjan}, equation (1.3)):
	\begin{equation*}\label{eq:Atoad}
	-A^\xi A^\xi X= R(X,\xi)\xi=-\frac{1}{4}\ad_\xi^2 X
	\end{equation*}
	for every $\xi\in\cal V_{\I}$. Therefore,   if $X$  is either   a $\alpha$-$A$-weight or $(\h\ad_\xi)^2X=\alpha(\xi)^2X$, we get 
	\begin{equation}\label{eq:A=vweight}
	(A^\xi)^2X=\alpha(\xi)^2X={\frac{1}{4}}\ad_\xi^2X.
	\end{equation}
	Recall that two roots $\alpha,\beta\in\Pi^v(\lie t^v)$ satisfying  $\alpha(\xi)^2=\beta(\xi)^2$ for all $\xi$ must satisfy $\alpha=\pm\beta$. In particular, 
	\begin{align*}
	\lie g_\alpha(\lie t^v)+\lie g_{-\alpha}(\lie t^v)&=\bigcap_{\xi\in\lie t^v}\ker\Big((\tfrac{1}{2}\ad_\xi)^2-\alpha(\xi)^2\id\Big),\\
	\cal H_\alpha(\lie t^v)+\cal H_{-\alpha}(\lie t^v)&= \bigcap_{\xi\in\lie t^v}\ker \Big((A^\xi)^2-\alpha(\xi)^2\id\Big).
	\end{align*}
	Thus,  $\cal H_\alpha(\lie t^v)+\cal H_{-\alpha}(\lie t^v)\subseteq \lie g_\alpha(\lie t^v)+\lie g_{-\alpha}(\lie t^v)$. Since $\cal H_\alpha(\lie t^v)+\cal H_{-\alpha}(\lie t^v)\subseteq \HH$, we conclude:
	\begin{gather*}
	\cal H_\alpha(\lie t^v)+\cal H_{-\alpha}(\lie t^v)=\HH\cap  (\lie g_\alpha(\lie t^v)+\lie g_{-\alpha}(\lie t^v)).\qedhere
	\end{gather*}
\end{proof}

In particular, all vertical roots appearing on horizontal vectors are the $A$-roots, i.e., 
\[X=X_0+\sum_{\alpha\in\Pi^A(\lie t^v)}X^\alpha=X_0+\sum_{\alpha\in\Pi^A(\lie t^v)}X_ \alpha, \]
where 
\[X_0\in \cal H_0(\lie t^v)=\cap_{\xi\in\lie t^v}\ker A^\xi=\HH\cap (\cap_{\xi\in\lie t^v}\ker\ad_\xi).\] 

Following Lemma \ref{lem:root1}, we define the projections $\pi_{\epsilon}(\lie t^v)\co \HH\to \lie g^\bb C$, $\epsilon=0,+,-$ by
\[\textstyle \pi_0(\lie t^v)(X_0+\sum X^\alpha)=X_0, \qquad   \pi_\pm(\lie t^v)(X^\alpha)=(X^\alpha)_{\pm\alpha}\in \lie g_{\pm\alpha}(\lie t^v). \]

So, $X=\pi_0(\lie t^v)(X)+\pi_+(\lie t^v)(X)+\pi_-(\lie t^v)(X)$ for every $X\in\HH$. 
Since $\lie g^\bb C$ is the complexification of $\lie g$, $\lie g^\bb C$ inherits two natural objects: a complex conjugation and the extension of the bi-invariant inner product $\lr{,}$ to a symmetric $\bb C$-bilinear form, also denoted by $\lr{,}$. We emphasize that we choose the $\bb C$-extension of $\lr{,}$ so both $A^\xi$ and $\ad_\xi$ are skew-symmetric. In particular:
\begin{gather*}\label{eq:product1}
\lr{X,Y}=\lr{X_0,Y_0}+\sum_{\alpha\in\Pi^A(\lie t^v)}\lr{X_\alpha,Y_{-\alpha}}=\lr{X_0,Y_0}+\sum_{\alpha\in\Pi^A(\lie t^v)}\lr{X^\alpha,Y^{-\alpha}}.  
\end{gather*}

We also observe that $\cal H_\epsilon(\lie t^v)$ are real subspaces, i.e., invariant by the complex conjugation, since they are (a sum of) the image of  real operators restricted to the kernel of real operators:
	\begin{align*}
\textstyle \cal H_\pm(\lie t^v)\cap (\lie g_\alpha(\lie t^v)\oplus \lie g_{-\alpha}(\lie t^v))=  (A^\xi\pm\h\ad_\xi)(\cal H_\alpha(\lie t^v)\oplus \cal H_{-\alpha}(\lie t^v))
\\ =\textstyle (A^\xi\pm\h\ad_\xi)\ker\Big((A^\xi)^2-\alpha(\xi)^2\id_{\HH}\Big). 
\end{align*}    
Analogously, $\cal H_0(\lie t^v)$ is the intersection of kernels of real operators. In particular,  $\cal H_0(\lie t^v)$ is orthogonal to $\cal H_+(\lie t^v)+\cal H_-(\lie t^v)$. We gather the notations/statements in the last paragraphs  as a lemma:

\begin{lem}
	Let $\lie t^v$ be a vertical abelian subagebra and $\pi_\epsilon(\lie t^v)$ be the projections defined by Lemma \ref{lem:root1}. Then:
	\begin{enumerate}[$(1)$]
		\item  $X=X_0+X_++X_-$ for every $X\in\HH$, where $X_\epsilon=\pi_\epsilon(\lie t^v)(X)$;
		\item $\cal H_\epsilon(\lie t^v)$ are real subspaces of $\lie g^\bb C$;
		\item $\cal H_0(\lie t^v)\perp  (\cal H_+(\lie t^v)+\cal H_-(\lie t^v)),$
	\end{enumerate}
\end{lem}

We state another  technical lemma to be used in the next section.

\begin{lem}\label{lem:tori0}
	Let $\lie t^v$ be a maximal vertical subalgebra and $\lie t\supseteq\lie t^v$ a maximal torus. Then, $\lie t$ decomposes orthogonally as $\lie t=\lie t^v\oplus \lie t'$  with $\lie t'\subseteq \cal H_0(\lie t^v)$.
\end{lem}
\begin{proof}
	Let $t\in\lie t$, $l\in\lie t^v$ and decompose $t$ in its vertical and horizontal components, $t=t^v+t^h$. On the one hand, $R(t,l)=\frac{1}{2}\ad_{[t,l]}=0$ . On the other hand, since fibers are totally geodesic,  $R(\cal H,l,\xi,\eta)=0$ for all $\xi,\eta\in\cal V_{\I}$. Thus
	\begin{equation*}
	0=R(t,l,\xi,\eta)=R(t^h,l,\xi,\eta)+R(t^v,l,\xi,\eta)=R(t^v,l,\xi,\eta). 
	\end{equation*}
	In particular, $R(t^v,l,l,t^v)=\frac{1}{4}||[t^v,l]||^2=0$. Since $l\in \lie t^v$ is arbitrary and $\lie t^v$ maximal, $ t^v\in\lie t^v$. Since $t^h=t-t^v\in \lie t$, we conclude that  $[t^h,l]=0$ for all $l\in \lie t^v$, thus  $t^h\in\cal H_0(\lie t^v)$.
\end{proof}

We claim that, if $\lie t''$ is an abelian subalgebra in $\cal H_0(\lie t^v)$, there is a maximal abelian subalgebra $\lie t'\subseteq \cal H_0(\lie t^v)$ such that $\lie t'\supseteq \lie t''$. Indeed, recall that any abelian subalgebra in a compact Lie group can always be extended to a maximal abelian subalgebra. Therefore,  $\lie t\oplus \lie t''$ can be extended to a maximal $\lie t$. Since $\lie t,\lie t^v$ are arbitrary in Lemma, we conclude that there is a $\lie t'\supseteq \lie t''$, such that $\lie t=\lie t^v\oplus \lie t'$.

\subsection{The bracket identity}\label{sec:brac2}
Fix a maximal vertical abelian subalgebra $\lie t^v$ and complete it to a maximal abelian subalgebra $\lie t=\lie t^v\oplus \lie t'$. We have:

%We proceed to a technical result.

\begin{prop}\label{prop:bracmaster}
%Let $\lie t=\lie t^v\oplus \lie t'$. Then, f
For every $X,Y\in \HH$ and $(\alpha,\alpha'), ~(\beta,\beta')\in\Pi(\lie t)$,
\begin{equation*}\label{eq:bracmaster}
	[(X_+)_{\alpha,\alpha'},(Y_-)_{\beta,\beta'}]=0.
	\end{equation*}
\end{prop}

We brake the proof into steps, stated as the next lemmas, and write $\pi_\pm(\lie t^v)(X^\alpha)=X^\alpha_\pm$. The first Lemma is a restatement of Theorem \ref{thm:tapp} taking into account the $\pi_\pm(\lie t^v)$-decomposition.

\begin{lem}\label{lem:brac0}
	Let $\xi\in\lie t^v$, $X\in \HH$. Then, for all $n,m\geq 0$,
	\begin{equation*}
	[\ad_{X}^m\ad_\xi X_-,\ad_\xi^{n+1} X_+]=0.
	\end{equation*}
\end{lem}
\begin{proof}
%	Recall that a triple $\{X,V, A\}\subseteq T_pG$ is  {good} if $\exp_p (tV(s))=\exp_p(sX(t))$ for all $s,t\in\bb R$, where $V(s),X(t)$ denote the Jacobi fields along $\exp(sX)$ and $\exp(tV)$ satisfying $V(0)=V$, $X(0)=X$ and $V'(0)=\cal A=X'(0)$, respectively. 
	Recall from the discussion in section \ref{sec:1} that $\{X,\xi,A^\xi X\}$ is a good triple for every $X\in\cal H$ and $\xi\in \cal V$.
	We apply Theorem \ref{thm:tapp} to it. 
	Let $X^\alpha$ be a $\alpha$-$A$-weight and observe that 
	\begin{gather*}
	\h\ad_\xi X^\alpha_\pm=\pm\alpha(\xi) X^\alpha_\pm,\quad\quad A^\xi (X^\alpha_++X^\alpha_-)=\alpha(\xi)X^\alpha.
	\end{gather*}
	Thus:
\begin{align*}\textstyle
B=(\h\ad_\xi -A^\xi) X &=
%\sum_{\alpha\neq 0}\left(\h\ad_\xi(X^\alpha_++X^\alpha_-)-A^\xi (X^\alpha_++X^\alpha_-)\right)\\=
\sum_{\alpha\in \Pi^{A}(\lie t^v)}\alpha(\xi)\left((X^\alpha_+-X^\alpha_-)- (X^\alpha_++X^\alpha_-)\right)\\
&=\sum_{\alpha\in \Pi^{A}(\lie t^v)}-2\alpha(\xi)X^\alpha_-=\ad_\xi X_-.
\end{align*}
Analogously, $\bar B=-\ad_\xi X_+$.
\end{proof}

Expanding the sum $X_\pm=\sum X^\alpha_\pm$, we get:

\begin{cor}\label{cor:brac0}
	Let $X \in \HH$, $\xi\in\lie t^v$. Then, for all $m\geq 0$ and $\beta\in\Pi^v(\lie t^v)$,
	\begin{equation*}
	\sum_{\alpha\in\Pi^A(\lie t^v)}\alpha(\xi)[\ad_{X}^m X_-^\alpha,X_+^\beta]=0.
	\end{equation*}
\end{cor}
\begin{proof}
	From Lemma \ref{lem:brac0}, we have
	\begin{equation*}
	0=[\ad_{X}^m\ad_\xi X_-,\ad_\xi^{n+1} X_+]=\sum_{\alpha,\beta}\alpha(\xi)\beta(\xi)^{n+1}	[\ad_{X}^m X_-^\alpha,X_+^\beta]
	\end{equation*}
	for every $n\geq 0$. Suppose $\xi$ is such that $\alpha(\xi)\neq \beta(\xi)$ for every pair of distinct $A$-roots $\alpha\neq \beta$. In this case, by taking enough values of $n$ we conclude that $\sum\alpha(\xi)[\ad_{X}^m X_-^\alpha,X_+^\beta]=0$ (recall that the determinant of the Vandermonde  matrix of a set of pairwise distinct values is non-zero). However, the set of such $\xi$'s is dense in $\lie t^v$, concluding the result for every $\xi$.
\end{proof}

\begin{lem}\label{lem:bracmaster}
	Let $X \in \HH$. Then, for all $l\geq 0$ and $\alpha,\beta\in\Pi^v(\lie t^v)$,
	\begin{equation*}
	[ X_-^\alpha,\ad_{X_0}^lX_+^\beta]=0.
	\end{equation*}
\end{lem}
\begin{proof}
	We use induction on $s$ in: for all $m\geq 0$,
	\begin{equation}\label{proof:brac0:0}
\sum_{\alpha\in\Pi^v(\lie t^v)}\alpha(\xi)[\ad_{X}^m X_-^\alpha,\ad_{X_0}^sX_+^\beta]=0.
	\end{equation}
%	Lemma \ref{lem:bracmaster} follows from \eqref{proof:brac0:0} and Claim \ref{claim:bracmaster2} below.
	Observe that \eqref{proof:brac0:0} holds for $s=0$ (Corollary \ref{cor:brac0}). As the induction hypothesis, we assume that \eqref{proof:brac0:0} holds for $s\leq k$ and compute $[\ad_X^mX^\alpha_-,\ad_{X_0}^{k+1}X^\beta_+]$ backwards:
	\begin{multline*}
	[\ad_X^{m+1}X^\alpha_-,\ad_{X_0}^{k}X^\beta_+]=[[X_0,\ad_X^{m}X^\alpha_-],\ad_{X_0}^{k}X^\beta_+]+[[X_-,\ad_X^{m}X^\alpha_-],\ad_{X_0}^{k}X^\beta_+]\\+[[X_+,\ad_X^{m}X^\alpha_-],\ad_{X_0}^{k}X^\beta_+]
	=\ad_{X_0}[\ad_X^{m}X^\alpha_-,\ad_{X_0}^{k}X^\beta_+]-[\ad_X^{m}X^\alpha_-,\ad_{X_0}^{k+1}X^\beta_+]\\
	+\ad_{X_-}[\ad_X^{m}X^\alpha_-,\ad_{X_0}^{k}X^\beta_+]-[\ad_X^{m}X^\alpha_-,[X_-,\ad_{X_0}^{k}X^\beta_+]]+[[X_+,\ad_X^{m}X^\alpha_-],\ad_{X_0}^{k}X^\beta_+].
	\end{multline*}
	That is,
	\begin{multline}
	\label{proof:brac0:1}
[\ad_X^{m}X^\alpha_-,\ad_{X_0}^{k+1}X^\beta_+]=	\ad_{X_0}[\ad_X^{m}X^\alpha_-,\ad_{X_0}^{k}X^\beta_+]-[\ad_X^{m+1}X^\alpha_-,\ad_{X_0}^{k}X^\beta_+]\\
	+\ad_{X_-}[\ad_X^{m}X^\alpha_-,\ad_{X_0}^{k}X^\beta_+]-[\ad_X^{m}X^\alpha_-,[X_-,\ad_{X_0}^{k}X^\beta_+]]-[[\ad_X^{m}X^\alpha_-,X_+],\ad_{X_0}^{k}X^\beta_+].
	\end{multline}
	In order to apply the induction hypothesis, we multiply both sides  by $\alpha(\xi)$ and sum  in $\alpha$.  It follows that the first three terms on the right-hand-side vanish.
	We deal with the last term in a separate claim. 
	\begin{claim}\label{claim:bracmaster1}
		$\sum\alpha(\xi)[X_+,\ad_X^{m}X^\alpha_-]=0$.
	\end{claim}
	\begin{proof}
		It is sufficient to prove that $\sum_{\alpha}\alpha(\xi)[X_+^\beta,\ad_X^{m}X^\alpha_-]=0$ for every $\beta\in\Pi^A(\lie t^v)$. We use induction on $s$ in: for every $r\geq 0$, 
		\begin{equation}\label{proof:lem:bracmaster1}
		\sum_{\alpha\in \Pi^A(\lie t^v)}\alpha(\xi)[\ad_X^rX_+^\beta,\ad_X^{s}X^\alpha_-]=0.
		\end{equation}
		The case $s=0$ is Corollary \ref{cor:brac0}. Assuming that \eqref{proof:lem:bracmaster1} holds for $s\leq k$, we have
		\begin{align*}
		\sum_{\alpha\in \Pi^A(\lie t^v)}\alpha(\xi)[\ad_X^{r}X_+^\beta,\ad_X^{k+1}X^\alpha_-]=&\ad_X\sum_{\alpha\in \Pi^A(\lie t^v)}\alpha(\xi)[\ad_X^{r}X_+^\beta,\ad_X^{k}X^\alpha_-]\\&-\sum_{\alpha\in \Pi^A(\lie t^v)}\alpha(\xi)[\ad_X^{r+1}X_+^\beta,\ad_X^{k}X^\alpha_-]=0.\qedhere
		\end{align*}
	\end{proof}

The 5th and the proof is completed once we observe that $[X_-^\alpha,\ad_{X_0}^{k}X^\beta_+]=0$ for any $\alpha\in\Pi^A(\lie t^v)$, provided \eqref{proof:brac0:0} holds for $s\leq k$.  However, since $X_0$ commutes with $\lie t^v$,  the term $[X_-^\alpha,\ad_{X_0}^{k}X^\beta_+]$ lies in $\lie g_{\beta-\alpha}(\lie t^v)$. Therefore, each term in  the induction hypothesis \eqref{proof:brac0:0}:
		\begin{equation*}
		\sum_{\alpha\in\Pi^A(\lie t^v)}\alpha(\xi)[X_-^\alpha,\ad_{X_0}^{k}X^\beta_+] =0.\label{proof:lem:bracmaster2}
		\end{equation*} 
lies in a different weight space. Since $\xi$ is arbitrary, each term must vanish, concluding the proof. 
%implying that $[X_-^\alpha,\ad_{X_0}^{k}X^\beta_+]=0$ for all $\alpha,\beta$. 
%		In particular, $[X_-,\ad_{X_0}^{k}X^\beta_+]=0$. The same argument completes the proof of the Lemma.
	\end{proof}
%	Now equation \eqref{proof:brac0:1} completes the induction and Claim \ref{claim:bracmaster2} the proof:
%	\begin{equation*}
%		\sum_{\alpha}\alpha(\xi)[\ad_X^{m}X^\alpha_-,\ad_{X_0}^{k+1}X^\beta_+]=\sum_{\alpha}\alpha(\xi)[\ad_X^{m+1}X^\alpha_-,\ad_{X_0}^{k}X^\beta_+]=0.\qedhere
%	\end{equation*}
%\end{proof}

\begin{proof}[Proof of Proposition \ref{prop:bracmaster}] We first prove the intermediate step $X=Y$ by following along the same lines as in the proof of Corollary \ref{cor:brac0}. Since $X_0$ is itself horizontal and does not influence $X_\pm$, we consider $X=t+X_++X_-$ where $t\in\lie t'$ can be chosen at our will. Given $\beta\in\Pi^A(\lie t^v)$, denote  $\Pi_\beta=\{\beta'\co \lie t'\to i\bb R~|~(\beta,\beta')\in\Pi(\lie t)\}$. Lemma \ref{lem:bracmaster} gives for all $l\geq 0$,
	\begin{equation*}
0=[ X_-^\alpha,\ad_{t}^lX_+^\beta]=\sum_{\beta'\in\Pi_\beta}\beta'(t)^l[ X_-^\alpha,(X_+^\beta)_{\beta,\beta'}]=\sum_{\beta'\in\Pi_\beta}\beta'(t)^l[ X_-^\alpha,(X_+^\beta)_{\beta,\beta'}].
	\end{equation*}
	Consider $t$ such that the values $\beta'(t)$, $\beta'\in\Pi_\beta$, are all distinct and nonzero. Taking enough values of $l$ gives
	\begin{equation}\label{proof:prop:bracmaster}
0=[X_-^\alpha,(X_+^\beta)_{\beta,\beta'}]=\sum_{\alpha'\in\Pi_{-\alpha}}[ (X_-^\alpha)_{-\alpha,\alpha'},(X_+^\beta)_{\beta,\beta'}].
	\end{equation}
Since $\alpha,\beta,\beta'$ are fixed, each term $[(X_-^\alpha)_{-\alpha,\alpha'},(X_+^\beta)_{\beta,\beta'}]$ lies in a different root space. Therefore, \eqref{proof:prop:bracmaster} implies that $[ (X_-^\alpha)_{-\alpha,\alpha'},(X_+^\beta)_{\beta,\beta'}]=0$ for all $\alpha,\alpha',\beta,\beta'$. Noting that  $(X_\pm^\alpha)_{\pm \alpha,\alpha'}=(X^\alpha)_{\pm \alpha,\alpha'}$, and writing $X_+=\sum X^\alpha_+$, $X_-=\sum X^\beta_-$, we get $[ (X_-)_{-\alpha,\alpha'},(X_+)_{\beta,\beta'}]=0$.

To proceed,  recall that $\lie g$ is the product of an abelian Lie algebra and a semi-simple Lie algebra. In particular, the root space $\lie g_{\alpha,\alpha'}(\lie t)$ is 1-dimensional and the brackets $[,]\co \lie g_{\alpha,\alpha'}(\lie t)\times\lie g_{\beta,\beta'}(\lie t)\to \lie g^\bb C$ are either zero, when $(\alpha+\alpha',\beta+\beta')\notin\Pi(\lie t)\cup\{(0,0)\}$, or are non-degenerate, i.e.,  $[x,y]=0$ only if $x=0$ or $y=0$. 

Let  $\pi_{\alpha,\alpha'}\co \lie g^\bb C\to\lie g_{\alpha,\alpha'}(\lie t)$ be the linear projection onto $\lie g_{\alpha,\alpha'}(\lie t)$ and denote $\pi_{\alpha,\alpha'}^\pm=\pi_{\alpha,\alpha'}\circ \pi_\pm(\lie t^v)$. 
Suppose  $(-\alpha+\beta,\alpha'+\beta')$ is a root (if not, $[(X_-)_{-\alpha,\alpha'},(Y_+)_{\beta,\beta'}]$ trivially vanishes). Since $[\pi_{-\alpha,\alpha'}^+(X),\pi_{\beta,\beta'}^-(X)]=0$  for every $X\in\HH$,
we conclude that $\HH=\ker \pi_{-\alpha,\alpha'}^+\cup \ker \pi_{\beta,\beta'}^-$. This is only possible if one of the kernels coincides with $\HH$. In particular,  for every pair $\alpha,\alpha',\beta,\beta'$, $[\pi_{-\alpha,\alpha'}^+(\cal H_+(\lie t^v)),\pi_{\beta,\beta'}^-(\cal H_-(\lie t^v))]=\{0\}$.
\end{proof}

%\begin{rem}\footnote{keep it or remove it?}
%	In the last two paragraphs we finally departed from the more general setting of a compact symmetric space. Indeed, all constructions up to this point hold (by considering a Cartan embedding), but it will fail to satisfy the dimensionality hypothesis of $\lie g_{\alpha,\alpha'}(\lie t)$. Nevertheless, the weaker inequality $[ (X_-)_{-\alpha,\alpha'},(X_+)_{\beta,\beta'}]=0$ is still valid.	
%\end{rem}

Proposition \ref{prop:bracmaster} shows, in particular,  that $\ker\pi^+_{\alpha,\alpha'}\cup \ker\pi^-_{-\alpha,-\alpha'}=\HH$. On the other hand, $\cal H_\epsilon(\lie t^v)$ are real spaces therefore 
%and the complex conjugate of $\lie g_{\alpha,\alpha'}(\lie t)$ is $\lie g_{-\alpha,-\alpha'}(\lie t)$, therefore 
$\pi^\pm_{\alpha,\alpha'}\neq \{0\}$ if and only if its complex conjugate, $\pi^{\pm}_{-\alpha,-\alpha'}$, satisfy $\pi^{\pm}_{-\alpha,-\alpha'}\neq \{0\}$. Putting together these two pieces of information,  we conclude that  $\ker\pi^+_{\alpha,\alpha'}\cup\ker\pi^-_{\alpha,\alpha'}=\HH$, i.e., each root (together with its negative) can appear as a component  of at most one of th two spaces $\cal H_+(\lie t^v)$, $\cal H_-(\lie t^v)$. These arguments derive a central property of:
\begin{equation}\label{eq:upsilon}
\Upsilon_\pm(\lie t^v)=\{(\alpha,\alpha')\in \Pi(\lie t^v)~|~\exists X\in\HH,~(X_+)_{\alpha,\alpha'}\neq 0  \},
\end{equation}
which is a main object in the next section. 
%Equivalently, $(\alpha,\alpha')\in \Upsilon_\pm(\lie t^v)$ if and only if $\pi^\pm_{\alpha,\alpha'}$ is not the zero map. 
We have shown:

\begin{cor}\label{cor:root1}
	$\Upsilon_+(\lie t)\cap \Upsilon_-(\lie t)=\emptyset$. In particular, $\cal H_+(\lie t^v)\cap \cal H_-(\lie t^v)=\{0\}$ and $\cal H_+(\lie t^v)\perp \cal H_-(\lie t^v)=\{0\}$.
\end{cor}

The orthogonality follows since $\cal H_\pm(\lie t^v)$ are real spaces.

\begin{rem}
	An important step  both in here (see Corollary \ref{cor:A}) and in \cite{ranjan}  is to show  that: 
	\begin{equation}\label{eq:A} 
	A_{X}Y=\h\Big([X_-,Y_-]-[X_+,Y_+]\Big)^v. 
	\end{equation}
	Such equality implies that $A^\xi X=\h\ad_\xi X_+ -\h\ad_\xi X_-$ for all $\xi\in\cal V_\I$. If one fix $\lie t^v$ and the same computations as in \cite{ranjan}, one can show that equation \eqref{eq:A} holds for $X=X^\alpha$, $Y=Y^\beta$, when $\alpha\neq \beta$. However, the information is lost for $\alpha=\beta$. We prove \eqref{eq:A} by using a much more refined decomposition.
\end{rem}

\subsection{The horizontal decomposition II}\label{sec:brac3}
We call an immersed subgroup
 $H\hookrightarrow G$  as \textit{$\cal V$-maximal} if its adjoint representation leaves $\cal V_{\I}$ invariant and it is transitive in the set of maximal vertical abelian subalgebras. That is,  $\Ad_H(\cal V_{\I})=\cal V_{\I}$ and, fixed $\lie t^v$,  every other maximal vertical abelian subalgebra is of the form $\Ad_h \lie t^v$. 
We recall a few points:
\begin{enumerate}
	\item if $L_{\I}$ is a subgroup, then $H=F_{\I}$ is $\cal V$-maximal;	\item  if $L_{\I}$ is an irreducible symmetric space which is not a Lie group, then a $\cal V$-maximal $H$ can be chosen as the subgroup whose Lie algebra is $\lie h=[\cal V_\I, \cal V_\I]$  (see Conlon \cite{conlon1972class} or Berestovskii-Nikonorov \cite[Lemma 7]{berestovskii-nikonorov}). It follows that $\lie h\subseteq \cal H_\I$;
	\item writing $h^*(\alpha,\beta)=(\alpha\circ \Ad_h^{-1},\beta\circ\Ad_h^{-1})$, 
\[	\Pi(\Ad_h\lie t)=\{h^*(\alpha,\beta)~|~(\alpha,\beta)\in\Pi(\lie t)\};\]
	\item $\lie g_{h^*(\alpha,\beta)}(\Ad_h\lie t)=\Ad_h (\lie g_{(\alpha,\beta)}(\lie t))$.
\end{enumerate}

Given a $\cal V$-maximal $H$, 
define the vector spaces 
\begin{align*}\label{eq:calHpm}
{\cal H}_\pm(\cal F)&=\sum_{h\in H}\cal H_\pm(\Ad_h\lie t^v),\\
\cal H_0(\cal F)&=\HH\cap ({\cal H}_+(\cal F)+{\cal H}_-(\cal F))^\bot.
\end{align*} 
It is clear that $\cal H_\epsilon(\cal F)$ is independent of $H$ and $\lie t^v$, however both are crucial in our prove of the commutation. We now state a main result:

\begin{theorem}\label{prop:Hpm}
Let $\cal F$ be a Ranjan foliation. Then 
\begin{enumerate}
	\item $\cal H_+(\cal F)\bot\cal H_-(\cal F)$;
	\item  $\cal H_+(\cal F)\cap \cal H_-(\cal F)=\{0\}$;
	\item  $[\cal H_+(\cal F),\cal H_-(\cal F)]{\,=\,}\{0\}$.	
\end{enumerate}
%	In particular, $A^\xi X=\h\ad_\xi X_+-\h\ad_\xi X_-$ for every $X\in\HH$. 
\end{theorem}

The current section is a preliminary step for Theorem \ref{prop:Hpm}, where we prove Proposition \ref{prop:Adtori}. Theorem \ref{prop:Hpm} is proved in section \ref{sec:proofHpm}. We fix an arbitrary maximal abelian subalgebras $\lie t^v\subseteq \lie t$ throughout.

\begin{prop}\label{prop:Adtori}
	For every $h\in H$, $\cal H_\pm(\Ad_h\lie t^v)=\Ad_h\cal H_\pm(\lie t^v)$.
\end{prop}

The proof of Proposition \ref{prop:Adtori} uses Proposition \ref{prop:bracmaster}, to control the set of $\rm{Ad}_h\lie t$-roots, and  the next three lemmas.
%Denote by $\Upsilon_\pm(\lie t)$ the collection of all roots appearing as nonzero components of elements in $\cal H_\pm(\lie t^v)$. 

%\begin{proof}
%	Since   $\cal H_\pm(\lie t^v)$ is a real space, it is closed under the natural complex conjugation in $\lie g^\bb C$. In particular,  $(\alpha,\beta)\in\Upsilon_\pm(\lie t)$ if and only if $(-\alpha,-\beta)\in\Upsilon_\pm(\lie t)$. Therefore, it is sufficient to prove that $\lr{(X_-)_{\alpha,\beta},(Y_+)_{-\alpha,-\beta}}=0$ for all $X,Y\in\HH$ (recall that $\lie g_{\alpha}(\lie t^v)\bot \lie g_{\beta}$ if $\beta\neq -\alpha$).
%	
%	Given a set of positive roots $\Sigma^+(\lie t)$ (see section II.1 in Knapp \cite{knapp2013lie}), $\lie t$ has a  generating set $\{H_{\alpha,\beta}\}_{(\alpha,\beta)\in\Sigma^+(\lie t)}$ satisfying 
%	\begin{equation*}
%	[(X_-)_{\alpha,\beta},(Y_+)_{-\alpha,-\beta}]=\lr{(X_-)_{\alpha,\beta},(Y_+)_{-\alpha,-\beta}}H_{\alpha,\beta}.
%	\end{equation*}
%for all $X,Y\in\HH$. But the left-hand-side is zero by Proposition \ref{prop:bracmaster}.
%\end{proof}

\begin{lem}\label{lem:AdgH0}
For every $h\in H$, $\cal H_0(\Ad_h\lie t^v)=\Ad_h\cal H_0(\lie t^v)$. 
\end{lem}
\begin{proof}
$\cal H_0(\lie t^v)=\HH\cap_{\xi\in\lie t^v}\ker \ad_\xi$. Therefore:
\begin{align*}
\Ad_h\cal H_0(\lie t^v)=& \, (\Ad_h\HH)\cap\!\left(\;\underset{\mathclap{\xi\in\lie t^v}}{\cap}\Ad_h\ker \ad_\xi\!\right)\!\\
=&\, \HH\cap \left(\;\underset{\mathclap{\xi\in\lie t^v}}{\cap}\ker \ad_{\rm{Ad}_h\xi}\!\right)
=\HH\cap\left(\;\;\underset{\mathclap{\;\;\;\xi\in\rm{Ad}_h\lie t^v}}{\cap\;}\;\;\ker \ad_{\xi}\right)\!\!. \qedhere
\end{align*}
\end{proof}

Let $\Upsilon(\lie t)=\Upsilon_+(\lie t^v)\cup \Upsilon_-(\lie t^v)$. Since $\Ad_g$ fixes $ \cal H_{\I}$,  $h^*\Upsilon(\lie t)=\Upsilon(\Ad_h\lie t)$. 
%Furthermore, Lemma \ref{lem:AdgH0} implies that, for any $X\in\oplus_{\alpha\neq 0}\cal H_\alpha(\lie t^v)$, $\Ad_hX\in \oplus_{\alpha\neq 0}\cal H_{h^*\alpha}(\Ad_h\lie t^v)$.
%
%In particular, if  $(\alpha,\beta)\in\Upsilon_+(\lie t)\cup\Upsilon_-(\lie t)$, then $h^*(\alpha,\beta)\in \Upsilon_+(\Ad_h\lie t)\cup\Upsilon_-(\Ad_h\lie t)$, i.e., $\Upsilon_+(\Ad_h\lie t)\cup\Upsilon_-(\Ad_h\lie t)=h^*(\Upsilon_+(\lie t)\cup\Upsilon_-(\lie t))$. 
Moreover:

\begin{lem}\label{prop:AdUpsilon}
For every $h\in H$, $\Upsilon_\pm(\Ad_h\lie t)=h^*(\Upsilon_\pm(\lie t))$.
\end{lem}
\begin{proof}
Given $(\alpha,\alpha')\in \Upsilon_+(\lie t)\cup\Upsilon_-(\lie t)$, we prove that
\[H^{\pm}_{\alpha,\alpha'}= \{h\in H~|~h^*(\alpha,\alpha')\in\Upsilon_\pm(\Ad_{h}\lie t)\}\]
are open subsets of $H$. Note that  $H^+_{\alpha,\alpha'}\cap H^-_{\alpha,\alpha'}=\emptyset$ (Corollary \ref{cor:root1}) and $H^+_{\alpha,\alpha'}\cup H^-_{\alpha,\alpha'}=H$, since $\Upsilon(\Ad_h\lie t^v)=h^*\Upsilon(\lie t^v)$. Since $H$ is connected, it is sufficient to prove that $H^\pm_{\alpha,\alpha'}$ are open.

Analogous to the proof of Proposition \ref{prop:bracmaster}, consider the projections 
\[\pi^\pm_{\alpha,\alpha'}(h)=\pi_{h^*\alpha,h^*\alpha'}\circ\pi^\pm(\Ad_h\lie t^v).\]

Now suppose that $(\alpha,\alpha')\in\Upsilon_+(\lie t^v)$. Then, there is $X$ such that $\pi^\pm_{\alpha,\alpha'}(\I)\neq 0$. Moreover,
\[\pi^\pm_{\alpha,\alpha'}(h)(X) = \pi^\pm_{\alpha,\alpha'}(h)(X^{h^*\alpha})=\frac{1}{2h^*\alpha(\Ad_h\xi)}((\textstyle \h\ad_{\Ad_h\xi}\pm A^{\Ad_h\xi})X)_{h^*\alpha,h^*\alpha'} \] 
for any $\Ad_h\xi\in\Ad_h\lie t^v$ such that $h^*\alpha(\Ad_h\xi)=\alpha(\xi)\neq 0$. By fixing $\xi\in\lie t^v$ we see that $\pi^\pm_{\alpha,\alpha'}(h)$ is  continuous as a family of operators with respect to $h$. Thus, if $\pi^\pm_{\alpha,\alpha'}(h)(X)\neq 0$, $\pi^\pm_{\alpha,\alpha'}(h')(X)\neq 0$ for $h'$ close to $h$, concluding that $H^\pm_{\alpha,\alpha'}$ is open.
\end{proof}
Given a set $S\subseteq \lie g$, denote $\cal{L}(S)$ as the subalgebra generated by $S$. Define the auxiliary spaces:
%For the proofs of Proposition \ref{prop:Adtori} and Theorem \ref{prop:Hpm}, define 
\begin{align*}\lie H_\pm(\lie t)&=  \cal L\left({\bigoplus_{(\alpha,\beta)\in\Upsilon_\pm(\lie t)}}\lie g_{(\alpha,\beta)}(\lie t)\right);\\
\lie H_0(\lie t)&= (\lie H_+(\lie t)+\lie H_-(\lie t))^\perp.\end{align*}

Proposition \ref{prop:bracmaster}, Corollary \ref{cor:root1} and invariance by complex conjugation guarantees that $[\lie H_+(\lie t),\lie H_-(\lie t)]=\{0\}$, $\lie H_+(\lie t)\bot \lie H_-(\lie t)$ and  $\lie H_+(\lie t)\cap \lie H_-(\lie t)=\{0\}$. 
Moreover, Lemma \ref{prop:AdUpsilon} implies that $\lie H_\pm(\Ad_h\lie t)=\Ad_h \lie H_\pm(\lie t)$ for all $h\in H$.

\begin{proof}[Proof of Proposition \ref{prop:Adtori}] Let $\pi_\pm(\lie t)\co \HH\to \lie H_\pm(\lie t)$ be the projections defined by the decomposition $\lie g=\lie H_+(\lie t)\oplus\lie H_-(\lie t)\oplus\lie H_0(\lie t)$. From Corollary \ref{cor:root1}, it follows that $\pi_\pm(\lie t)(\HH)=\cal H_\pm(\lie t^v)$.
From Lemma \ref{prop:AdUpsilon}, $\pi_{\pm}(\Ad_h\lie t)=\Ad_h\circ \pi_{\pm}(\lie t)\circ \Ad_{h^{-1}}$,  $h\in H$. Therefore,
\begin{align*}\cal H_\pm(\Ad_h\lie t^v)&=\pi_\pm(\Ad_h\lie t)(\HH)=\Ad_h(\pi_\pm(\lie t)(\Ad_{h^{-1}}\HH))\\&=\Ad_h(\pi_\pm(\lie t)(\HH))=\Ad_h(\cal H_\pm(\lie t^v)).  \qedhere\end{align*}
\end{proof}

Proposition \ref{prop:Adtori} gives a new characterization of  $\cal H_\pm(\cal F)$:  $\cal H_\pm(\cal F)$ is the smallest $\Ad_H$-invariant subset containing $\cal H_\pm(\lie t^v)$.

\subsection{Proof of Theorem \ref{prop:Hpm}}\label{sec:proofHpm}
In order to prove Theorem \ref{prop:Hpm}, fix $\lie t^v$ and observe from  Proposition \ref{prop:Adtori} and Jacobi identity that $[\cal H_+(\cal F),\cal H_-(\cal F)]=\{0\}$ if and only if $[\Ad_H\cal H_+(\lie t^v), \cal H_-(\lie t^v)]=\{0\}$. 
Moreover, every element in $H$ can be written as $e^\theta$ for some $\theta\in \lie h$, since $\lie g$ is compact and $H$ is connected.
%$\lie h$ is a compact Lie algebra, since it is a subalgebra of a compact Lie algebra, and $H$ is connected by definition. Thus 
Thus, a power series  argument guarantees  that $[\cal H_+(\cal F),\cal H_-(\cal F)]=\{0\}$ if and only if  $[\ad_\theta^k\cal H_+(\lie t^v),\cal H_-(\lie t^v)]=\{0\}$ for every $\theta\in\lie h$  and $k\geq 0$.

Our next aim is to show, by brute force, that $[\ad_\theta^k\cal H_+(\lie t^v),\cal H_-(\lie t^v)]=\{0\}$. We start by studying the elements of $\lie h$.

If $L_{\I}$ is a subgroup, we can take $\lie h=\cal V_I$. If $L_\I$ is an irreducible symmetric space which is not a group, we chose $\lie h=[\cal V_{\I},\cal V_{\I}]$. However, $L_{\I}$ might be reducible, so we write $\cal V_{\I}=\bigoplus\Delta_i$, where $\exp(\Delta_i)$ are locally irreducible symmetric spaces. An standard argument shows that the sum of the aforementioned options produces a $\cal V$-maximal group.

\begin{lem}\label{claim:Hpm1}
	If $i\neq j$, then $[\Delta_i,\Delta_j]=0$. In particular, $H$ is $\cal V$-maximal, for $\lie h=\sum\lie h_i$ where $\lie h_i=\Delta_i$ whether $\Delta_i$ is a subalgebra or $\lie h_i=[\Delta_i,\Delta_i]$ otherwise.
\end{lem}
\begin{proof}
	Since $L_{\I}$ is totally geodesic and is locally isometric to a metric product $\exp(\Delta_0)\times\cdots\times\exp(\Delta_s)$, the curvature tensor of $G$ at the identity satisfies $R(\Delta_i,\Delta_j)=\{0\}$. Therefore, $\lr{R(\xi,\eta)\eta,\xi}=\frac{1}{4}\|[\xi,\eta]\|^2=0$ for all $\xi\in\Delta_i$, $\eta\in\Delta_j$. In particular,  $\lie h=\sum \lie h_i$ integrates a subgroup which is, up to covering, a product $H=\tilde H_0\times\cdots\times \tilde H_s$. 
	
	To see that $H$ is transitive in the set of maximal vertical abelian subalgebras, note that a maximal abelian subalgebra of $\cal V_{\I}$ splits as $\lie t^v=\bigoplus \lie t^v\cap \Delta_i$ (one can use arguments as in Lemma \ref{lem:tori0}, for example). Thus, since each $\tilde H_i$ acts transitively on the set of abelian subalgebras of $\Delta_i$, $H$ acts transitively on the set of maximal abelian subalgebras of $\cal V_{\I}$.
\end{proof}
Whenever $\Delta_i$ is not a subalgebra, Besrestovskii--Nikonorov \cite[Lemma 7]{berestovskii-nikonorov} guarantees that $\exp(\Delta_i\oplus \lie h_i)$ is the full subgroup of isometries of $\exp(\Delta_i)$ and that $(\Delta_i\oplus \lie h_i,\lie h_i)$ is a symmetric pair (note that $[\Delta_i,\lie h_i]\subseteq \Delta_i$, since $\Delta_i$ is a Lie triple system -- see e.g. Helgason \cite{helgasondifferential}). In this case, $\lie h_i$ is  horizontal: it is orthogonal to $\Delta_j$, $j\neq i$, since $\lr{[\Delta_i,\Delta_i],\Delta_j}=\lr{\Delta_i,[\Delta_j,\Delta_i]}=\{0\}$;  and orthogonal to $\Delta_i$ since $[\Delta_i,\lie h_i]\subseteq \Delta_i$
%\footnote{details?}. 

We decompose $\lie h=\lie h^v\oplus \lie h^h$ in its vertical and horizontal components, denoting by $\Delta^\vee$ (respectively, $\Delta^\wedge$)  the sum of the $\Delta_i$-components which are subalgebras (respectively, which are not subalgebras).

Observe that $[\lie h^h,\lie h^v]=\{0\}$ and decompose $\theta\in \lie h $ in its horizontal and vertical components, $\theta=Z+\zeta$. We proceed by induction on $m$ to show that: for all $m,n\geq 0$,
\begin{equation}\label{proof:Hpm}
[\ad_{\zeta}^m\ad_Z^n\cal H_+(\lie t^v),\cal H_-(\lie t^v)]=\{0\}.
\end{equation} 
First we show that \eqref{proof:Hpm} holds for $m=0$ (Claim \ref{claim:Hpm2}), then, assuming that \eqref{proof:Hpm} holds for $m\leq k$, we show that it holds for $m=k+1$.

\begin{claim}\label{claim:Hpm2}
	$[\ad_Z^n\cal H_+(\lie t^v),\cal H_-(\lie t^v)]=\{0\}$ and $\ad_Z^n \cal H_+(\lie t^v)\perp \cal H_-(\lie t^v)$ for all $n\geq 0$.
\end{claim}
\begin{proof}
	Let $Z\in\lie h^h$ and write $Z_\epsilon=\pi_\epsilon(\lie t^v)(Z)\in\cal H_\epsilon(\lie t^v)$. We choose $\lie t=\lie t^v\oplus \lie t'$ such that $Z_0\in\lie t'$. 
	Since $[\lie H_+(\lie t),\cal H_-(\lie t^v)]=\{0\}$, $\lie H_+(\lie t)\perp\cal H_-(\lie t^v)$ and $\cal H_+(\lie t^v)\subseteq \lie H_+(\lie t)$, it is sufficient to show that $\lie H_+(\lie t)$ is $\ad_Z$-invariant. But, $\ad_{Z_0}(\lie H_+(\lie t))
	\subseteq \lie H_+(\lie t)$, since $Z_0\in\lie t$ and $\lie H_+(\lie t)$ is a sum of weight spaces;  $\ad_{Z_+}(\lie H_+(\lie t))
	\subseteq \lie H_+(\lie t)$,  since $Z_+\in \lie H_+(\lie t)$; $\ad_{Z_-}(\lie H_+(\lie t))=\{0\}
	\subseteq \lie H_+(\lie t)$,  by Proposition \ref{prop:bracmaster}.
\end{proof}
From now on, we assume $\lie h^v\neq\{0\}$  and proceed to technical steps.
\begin{claim}\label{claim:Hpm0}
	$X_\pm^v\in\Delta^\wedge$. In particular, $\ad_\zeta^m(\cal H_\pm(\lie t^v))\subseteq\HH$,
\end{claim}
\begin{proof}
		Let $X_\pm^\alpha=(X^{\alpha})_+\in\cal H_\pm(\lie t^v)$ be the $\cal H_\pm(\lie t^v)$ component of an $\alpha$-$A$-weight. Recall that $\Delta^\vee$ is a subalgebra and $[\Delta^\vee,\Delta^\wedge]=0$. Therefore, $\ad_{\Delta^\vee}$ preserves the decomposition $\Delta^\vee\oplus\Delta^\wedge\oplus \HH$. Moreover, if 		 $\alpha(\xi)\neq 0$ for some $\xi\in\Delta^\vee\cap \lie t^v$, 
		\begin{equation*}
		\HH\ni (\alpha(\xi)\pm \textstyle{\h}\ad_{\xi})X^\alpha=(\alpha(\xi)\pm\h\ad_{\xi})(X^\alpha_++X^\alpha_-)=2\alpha(\xi)X^\alpha_\pm.
		\end{equation*} 
		Therefore,  $(X_+^\alpha)^v\neq 0$ only if $\alpha(\xi)=0$ for every $\xi\in \Delta^{\vee}\cap \lie t^v$. 
		Since $\cal H_0(\lie t^v)\subseteq \HH$, we conclude that $(X_+^\alpha)^v\neq 0$ only if $\alpha(\xi')=1$ for some $\xi'\in\lie t^v\cap \Delta^\wedge$. Thus, 
		\[\lr{X^\alpha_\pm,\Delta^\vee}=\lr{\pm\tfrac{1}{2}\ad_{\xi'}X^\alpha_\pm,\Delta^\vee}=-\tfrac{1}{2}\lr{X^\alpha_\pm,\ad_{\xi'}\Delta^\vee}=0. \]
		The Claim is concluded by observing that $\Delta^\vee,\Delta^\wedge$ are real spaces.
\end{proof}

The next claim is a common induction step in the next proofs.

\begin{claim}\label{claim:Hpm induction}
	Suppose that $X'\in \HH\cap (\cal H_0(\lie t^v)+\cal H_+(\lie t^v))$ and $[X',\cal H_-(\lie t^v)]=0$. Then $\ad_\zeta X'\in \HH\cap (\cal H_0(\lie t^v)+\cal H_+(\lie t^v))$ and $[\ad_\zeta X',\cal H_-(\lie t^v)]=0$.
\end{claim}
\begin{proof}
	Note that $\ad_\zeta X'\in \HH\cap (\cal H_0(\lie t^v)+\cal H_+(\lie t^v))$:
	\begin{equation}\label{eq:commuting}
	\lr{\ad_\zeta X',\cal H_-(\lie t^v)}=\lr{\zeta,[X',\cal H_-(\lie t^v)]}=0.
	\end{equation}
	
	Therefore, since $\ad_\zeta X'$ has no $\cal H_-(\lie t^v)$-component,
	%\footnote{Due to possibly ambiguity, we clarify that the subindex $_0$ here refer to the $\cal H_\epsilon(\lie t^v)$-decomposition, not to the $\lie t$-root decomposition.}, 
	\begin{equation*}
	[\ad_\zeta X',Y_-]=[(\ad_\zeta X')_0,Y_-]=[(\ad_\zeta X'_0)_0,Y_-]+[(\ad_\zeta X'_+)_0,Y_-].
	\end{equation*}
	We show that both terms are zero. Since $\lie t^v\cap \Delta^\vee$ is a maximal torus, we can write $\zeta=\zeta_0+\sum\zeta_\alpha$, where $\zeta_0\in\lie t^v\cap \Delta^\vee$ and $\zeta_\alpha\in\lie g_\alpha(\lie t^v)$, $\alpha\neq 0$.
	On the other hand, $\ad_{\zeta_\alpha}X_0'\in\lie g_\alpha(\lie t^v)$. Thus, $(\ad_\zeta X'_0)_0=(\ad_{\zeta_0} X'_0)_0=0$, since  $\ad_{\lie t^v}\cal H_0(\lie t^v)=0$.
	
	On its turn, the second term belongs to $\lie H_-(\lie t)$, for $\lie t'$ such that $ (\ad_\zeta X'_+)_0\in \lie t'$. On the other hand, by replacing $X'$ by $X'_+$ on equation \eqref{eq:commuting}, we have
	$[(\ad_\zeta (X'_+))_0,Y_-]=[\ad_\zeta X'_+,Y_-]$. Thus,
	\begin{equation*}
	\lr{[\ad_\zeta X'_+,Y_-],\lie H_-(\lie t)}=\lr{[\ad_\zeta Y_-,X'_+],\lie H_-(\lie t)}=\lr{\ad_\zeta Y_-,[X'_+,\lie H_-(\lie t)]}=0.
	\end{equation*}
	Since $\lie H_-(\lie t)$ is a real space, we conclude that  $[\ad_\zeta X'_+,Y_-]=[\ad_\zeta X',Y_-]=0$. 
\end{proof}

In particular, $\ad_\zeta^m X_+\in\HH\cap(\cal H_0(\lie t^v)+\cal H_+(\lie t^v))$ for $m\geq 1$.

\begin{claim}\label{claim:Hpm2.66}
	For any $X\in\HH$, there is a decomposition $X_+=\overline X_++X'$ where $Y_+\in \HH\cap \cal H_+(\lie t^v)$ and $\ad_\zeta X'=0$. Moreover, $\ad_Z^n \overline X_+\in \HH\cap (\cal H_0(\lie t^v)+\cal H_+(\lie t^v))$.
\end{claim}
\begin{proof}
%	Let $\cal H_+^\zeta=\oplus_{m\geq 0}^\infty\ad_\zeta^m\cal H_+(\lie t^v)$. It follows that $\cal H^\zeta_+$ is $\ad_\zeta$-invariant and 
%	\[\cal H^\zeta_+= \ker \ad_\zeta|_{\cal H_+(\lie t^v)}+\ad_\zeta(\cal H^\zeta_+) \subseteq \ker \ad_\zeta|_{\cal H_+(\lie t^v)}+\HH\cap (\cal H_0(\lie t^v)+\cal H_+(\lie t^v)) \]
%	(see Claims \ref{claim:Hpm0} and \ref{claim:Hpm induction}). 
	
	We prove the Claim for each component $X^\alpha_+=(X^\alpha)_+$ of $X_+$, considering separate cases: if $\alpha(\lie t^v\cap \Delta^\vee)\neq \{0\}$, $\overline X_+=X^\alpha_+$ and $X'=0$ satisfies the desired conditions.
	
	On the other hand, since $[\Delta^\vee,\Delta^\wedge]=\{0\}$:
	\[\ad_{\xi'}\ad_\zeta X_+^\alpha=\ad_\zeta\ad_{\xi'}X^\alpha_+=2\alpha(\xi')\ad_\zeta X^\alpha_+ \]
	for every $\xi'\in\lie t^v\cap \Delta^\wedge$.
	Note that $\ad_\zeta X^\alpha_+ \in \cal H_+(\lie t^v)$ (Claim \ref{claim:Hpm induction} plus the fact that $\ad_{\xi'}\ad_\zeta X^\alpha_+\neq 0 $ for some $\xi'\in\lie t^v$), thus, $\ad_\zeta$ preserves the eigenspaces $V_\lambda$ of $\ad_{\xi'}$. Thus, $X^\alpha_+$ can be decomposed as $\overline X_++X'$, where $\overline X_+\in\ad_\zeta(V_\lambda)$ and $X'\in\ker\ad_\zeta$.
	
	The second statement follows since $\ad_Z$ preserves $\HH$ and Claim \ref{claim:Hpm2}:
	\begin{equation}
	\lr{\ad_Z^n \overline X_+,\cal H_-(\lie t^v)}=\lr{Z,[\ad_Z^{n-1}\overline X_+,\cal H_-(\lie t^v)]}=0.\qedhere
	\end{equation}
\end{proof}

\begin{proof}[Proof of Theorem \ref{prop:Hpm}]
	The main item in Theorem \ref{prop:Hpm} is the third item, from where we start the proof. To this aim, we proceed by induction on $m\geq 1$ on:
		\begin{align}\label{proof:Hpm:0}
	\ad_\zeta^m\ad_Z^n X_+\in\HH\cap(\cal H_0(\lie t^v)+\cal H_+(\lie t^v)),\\
	[\ad_{\zeta}^{m}\ad_Z^nX_+,\cal H_-(\lie t^v)]=0, \label{proof:Hpm:1} 
	\end{align}
	for every $n\geq 0$. From Claim \eqref{claim:Hpm2.66}, induction on \eqref{proof:Hpm:0}, \eqref{proof:Hpm:1} is equivalent to induction on
		\begin{align*}
	\ad_\zeta^m\ad_Z^n \overline X_+\in\HH\cap(\cal H_0(\lie t^v)+\cal H_+(\lie t^v)),\\
	[\ad_{\zeta}^{m}\ad_Z^n\overline X_+,\cal H_-(\lie t^v)]=0, 
	\end{align*}
	for $\bar X_+\in\HH\cap \cal H_{+}(\lie t^v)$. The last induction follows from Claim \ref{claim:Hpm induction}, concluding item \textit{(3)} in Theorem \ref{prop:Hpm}. 
	
	Since $\cal H_{\pm}(\cal F)$ are real spaces, item \textit{(2)} follows from item \textit{(1)}. Moreover, item \textit{(1)} holds once it is proved that $\ad_\zeta^m\ad^n_Z X_+\perp \cal H_-(\lie t^v)$ for every $m,n\geq 0$. The case $m=0$ is in Claim \ref{claim:Hpm2} and $m\geq 1$ is \eqref{proof:Hpm:0}.
\end{proof}

As an application of Theorem \ref{prop:Hpm}, we characterize the $A$-tensor.

\begin{lem}\label{cor:A}
Let $\cal F$ be a Ranjan foliation and write $X=X_0+X_++X_-$, where $X_\epsilon\in\cal H_\epsilon(\cal F)$. Then $A^\xi X=\h\ad_\xi(X_+-X_-).$	In particular,
\begin{equation}\label{eq:Aproved}
 A_XY=\h\left([X_-,Y_-]-[X_+,Y_+] \right)^v. 
\end{equation}\end{lem}
\begin{proof}
	Fix $X\in\cal H_{\I}$ and $\xi\in\cal V_{\I}$. 	
	Let $\lie t^v$ be a maximal abelian vertical subalgebra containing $\xi$. Using the linearity of $A^\xi$, we divide the proof into two cases: $(i)$ $X$ has no $\cal H_0(\lie t^v)$-component; $(ii)$  $X\in\cal H_0(\lie t^v)$. 
	
	Denote by 
	\begin{gather*}
	\pi^\epsilon\co \cal H_+(\lie t^v)+\cal H_-(\lie t^v)+\cal H_0(\lie t^v)\to \cal H_\epsilon(\cal F)\\ \pi^\epsilon(\lie t^v)\co \HH \to \cal H_\epsilon(\lie t^v) 
	\end{gather*}
	the respective orthogonal projections. Supposing that $\pi^0(\lie t^v)(X)=0$, we have
	\[\textstyle A^\xi X= A^\xi(\pi^+(\lie t^v)(X)+\pi^-(\lie t^v)(X))=\h\ad_\xi (\pi^+(\lie t^v)(X))-\h\ad_\xi(\pi^-(\lie t^v)(X)).\]
	On the other hand, since $\cal H_\pm(\lie t^v)\subseteq \cal H_\pm(\cal F)$, $\pi^\pm\circ \pi^\pm(\lie t^v)=\pi^\pm(\lie t^v)$ and $\pi^{\mp}\circ\pi^{\pm}(\lie t^v)=0$.
		
	Now, assume that $X\in\cal H_0(\lie t^v)$. We claim that $\pi^\pm(\cal H_0(\lie t^v))\subseteq \cal H_0(\lie t^v)$. By following Claim \eqref{claim:Hpm2} and \eqref{proof:Hpm:0}, we conclude that an element $X_\pm$ is composed by components lying either in $\cal H_0(\lie t^v)$ (components with $m\geq 1$ in \eqref{proof:Hpm:0}) or in $\sum_{\lie t\supseteq \lie t^v}\lie H_\pm(\lie t)$ (Claim \ref{claim:Hpm2}). I.e., $\cal H_+(\cal F)$ decomposes as
	\begin{equation}\label{eq:HFbounds}
	\cal H_\pm(\cal F)= (\cal H_\pm(\cal F)\cap \cal H_0(\lie t^v))\oplus\left( \cal H_\pm(\cal F)\cap \sum_{\lie t\supseteq \lie t^v}\lie H_\pm(\lie t)\right).
	\end{equation}
	Since the second space space is orthogonal to $\cal H_0(\lie t^v)$, $\pi^0(\cal H_0(\lie t^v))$ do not have components on it, thus concluding that $\pi^\pm(\cal H_0(\lie t^v))\subseteq \cal H_0(\lie t^v)$.

	With $A^\xi X=\h\ad_\xi(X_+-X_-)$ at hand, equation \eqref{eq:Aproved} is straightforward:
	\begin{align*}
	-2\lr{A_XY,\xi}&= 2\lr{A^\xi X,Y }=\lr{\ad_\xi (X_+-X_-),Y}=\lr{\ad_\xi (X_+-X_-),Y_++Y_-}\\
	&=\lr{\xi,[X_+,Y_++Y_-]-[X_-,Y_++Y_-]}=\lr{\xi,[X_+,Y_+]-[X_-,Y_-]}.\qedhere
	\end{align*}
\end{proof}

\subsection{Proof of Theorem \ref{thm:grove}}\label{sec:proof}
We now have all elements to prove Theorem \ref{thm:grove}. To simplify notation, we denote $\cal H_\epsilon(\cal F)=\cal H_\epsilon$.

As pointed out in section \ref{sec:pi1}, it is sufficient to prove Theorem \ref{thm:grove} for irreducible foliations. 
In this case, Theorem \ref{thm:A} guarantees that $\cal V_\I$ is spanned by the image of the $A$-tensor. Thus Lemma \ref{cor:A} gives:
	\begin{equation*}
	\cal V_\I\subseteq \HH+ [\cal H_+ ,\cal H_+ ]+[\cal H_- ,\cal H_- ].
	\end{equation*}
	In particular:
\begin{equation}\label{eq:V}
	\lie g^\bb C \subseteq \cal H_0 +\cal L(\cal H_+ )+\cal L(\cal H_- ), 
\end{equation}
	where $\cal L(S)$ is the Lie algebra generated by $S\subseteq \lie g^\bb C$. Moreover:
\begin{claim}\label{claim:H0}
	$[\cal H_0 ,\cal H_\epsilon ]\subseteq \cal H_\epsilon $.
\end{claim}
\begin{proof}
	Observe that 
	\begin{align*}
	\cal H_0(\cal F)&=\HH\cap ({\cal H}_+(\cal F)+{\cal H}_-(\cal F))^\bot\\
	&=\HH\cap \bigcap_{h\in H}({\cal H}_+(\Ad_h\lie t^v)+{\cal H}_-(\Ad_h\lie t^v))^\bot=\bigcap_{h\in H}\cal H_0(\lie t^v)\\
	&=\{X\in\HH~|~ \ad_\xi X=0,~\forall\xi\in\cal V_I\}. \end{align*}
	In particular,  $\cal H_0 $ is a subalgebra: let $X,Y\in\cal H_0 $, $\xi\in\cal V_\I$, then
	\[\ad_\xi[X,Y]=[\ad_\xi X,Y]+[X,\ad_\xi Y]=0. \]
	Moreover, $[\cal H_0 ,\cal{V}_\I]=\{0\}$, thus $\cal H_0^\bot$, $\HH$  and therefore $\cal H_0 ^\perp\cap \HH$ are invariant under $\ad_{\cal H_0 }$. In particular,  $[\cal H_0 ,\cal H_+ +\cal H_- ]\subseteq \HH\cap(\cal H_+ +\cal H_-) $. On the other hand,
	\[\lr{[\cal H_0 ,\cal H_\pm ],\cal H_\mp  }=-\lr{\cal H_0 ,[\cal H_\pm ,\cal H_\mp ] }=0.\qedhere   \] 
\end{proof}

\begin{cor}
	$\cal L(\cal H_\pm )$ is an ideal. 
\end{cor}
\begin{proof}
	Let $Z\in\lie g^\bb C$. Then $Z=Z_0+Z_++Z_-$, where $Z_\epsilon\in\cal L(\cal H_\epsilon )$ (equation \eqref{eq:V}). But
	$[Z_0,~\cal L(\cal H_\pm)]\subseteq \cal L(\cal H_\pm) $ by Claim \ref{claim:H0}; $[Z_\pm,\cal L(\cal H_\pm)]\subseteq \cal L(\cal H_\pm)$ by the {definition of }$\cal L(\cal H_\pm)$; and
	$[Z_\mp,\cal L(\cal H_\pm)]=\{0\}$ by {Theorem \ref{prop:Hpm}} and Jacobi identity.
\end{proof}

Since $\cal L(\cal H_\pm)$ are real subspaces, their real parts are ideals of $\lie g$, which we denote by $\cal L(\cal H_\pm)$ as well. In particular $G$ decomposes as $G= G_+{\times }G_-{\times }G_0$, where $G_\epsilon$ is the subgroup whose Lie algebra is
\begin{align*}
\lie g_\pm & = \cal L(\cal H_\pm)\cap (\cal L(\cal H_+)\cap \cal L(\cal H_-))^\perp,\\
\lie g_0& = (\lie g_++\lie g_-)^\perp+\cal L(\cal H_+)\cap \cal L(\cal H_-).
\end{align*}
By observing that $\cal H_\pm\perp \cal L(\cal H_\mp)$, we conclude that $\cal H_++\cal H_-\perp \lie g_0$. 

We claim that $A_{X}Y=\h([X_-,Y_-]-[X_+,Y_+])^v$, where $X_\epsilon\in\lie g_\pm$. Denote $\pi^\epsilon(Z)$ the $\cal H_\epsilon $-component of $Z\in\HH$. Since $[\pi^0(Z),\xi]=0$ for all $\xi\in\cal V_\I$, we conclude that  $[\pi^0(Z)_\epsilon,\xi]=[\pi^0(Z)_\epsilon,\xi_\epsilon] =0$ for $\epsilon=0,+,-$. Thus,
\begin{align*}
\lr{[X_-,Y_-],\xi}=&\lr{[\pi^-(X)+\pi^0(X)_-,\pi^-(Y)+\pi^0(Y)_-],\xi}\\
=&\lr{[\pi^-(X),\pi^-(Y)+\pi^0(Y)_-],\xi}+\lr{\pi^-(Y)+\pi^0(Y)_-,[\xi,\pi^0(X)_-]}\\
=&\lr{[\pi^-(X),\pi^-(Y)],\xi}+\lr{\pi^-(Y)+\pi^0(Y)_-,[\xi,\pi^0(X)_-]}+\lr{\pi^-(X),[\pi^0(Y)_-,\xi]}\\
=&\lr{[\pi^-(X),\pi^-(Y)],\xi}.
\end{align*}
Analogously, $\lr{[X_+,Y_+],\xi}=\lr{[\pi_+(X),\pi(Y)_+],\xi}$. The claim now follows from Lemma \ref{cor:A}.

The proof is  almost finished and follows from Munteanu--Tapp Corollary \ref{cor:TappMunteanu}.
%
%\begingroup
%\def\thetheorem{\ref{cor:TappMunteanu}}
%\begin{cor}[Munteanu--Tapp \cite{tappmunteanu2}, Corollary 4.2] 
%	Let $\mathcal F$ be a Riemannian foliation with totally geodesic leaves on a connected Lie group $G$ with bi-invariant metric. If $A^\xi X=\h\ad_\xi X$, for all $\xi\in \cal V_\I,X\in\cal H_\I$, then $\cal V_\I$ is a subalgebra and $\cal F$ is the foliation defined by the left cosets of the subgroup  whose subalgebra is $\cal V_\I$.
%\end{cor}
%\addtocounter{theorem}{-1}
%\endgroup
%
It is only left to produce an isometry of $G$ whose resulting foliation satisfy $A^\xi X=\h\ad_\xi X$.
Let $\Phi\co G\to G$ be the isometric involution defined by 
\[\Phi(g_+,g_-,g_0)=(g_+,g_-^{-1},g_0), \]
and consider 
\[\tilde{\cal F}=\{\Phi(L)~|~L\in \cal F \}.\] 
Denote $\tilde Z=d\Phi(Z)$, $\tilde{\cal V}_\I=d\Phi(\cal V_\I)$ and by $\tilde A$ the $A$-tensor of $\tilde{\cal F}$.  Observe that $d\Phi_e(Z_\pm)=\pm Z_\pm$ and recall that $\tilde A_{\tilde X}\tilde Y=d\phi(A_XY)$. We have:
\begin{align*}
\tilde A_{\tilde X}\tilde Y&=d\Phi(A_XY)=\h\left(d\Phi([X_-,Y_-]-[X_+,Y_+]) \right)^{\tilde{\cal V}_\I}\\
&=\h\left(-[X_-,Y_-]-[X_+,Y_+] \right)^{\tilde{\cal V}_\I}=-\h\left([\tilde X_-,\tilde Y_-]+[\tilde X_+,\tilde Y_+] \right)^{\tilde{\cal V}_\I},
\end{align*}
where the third equality follows since $[X_\pm,Y_\pm]\in \lie g_\pm$ and the last since
\[[\tilde X_\pm,\tilde Y_\pm]=[\pm X_\pm,\pm Y_\pm ]=[X_\pm,Y_\pm]. \]

In particular, the respective dual tensor is given by 
\[\tilde A^{\tilde \xi}\tilde X=\h\ad_{\tilde \xi} \tilde X.
\]
Corollary \ref{cor:TappMunteanu} guarantees that $\tilde{\cal V}_I$ is a subalgebra and that $\tilde{\cal F}$ is the coset fibration defined by the subgroup integrated by $\tilde{\cal V}_\I$, completing the proof.
\halmos

\bibliographystyle{amsplain}
\bibliography{bibgrove}

\end{document}